\documentclass[12pt]{amsart}
\usepackage{times}
\begin{document}
\numberwithin{equation}{section}

\def\1#1{\overline{#1}}
\def\2#1{\widetilde{#1}}
\def\3#1{\widehat{#1}}
\def\4#1{\mathbb{#1}}
\def\5#1{\frak{#1}}
\def\6#1{{\mathcal{#1}}}

\newcommand{\de}{\partial}
\newcommand{\R}{\mathbb R}
\newcommand{\al}{\alpha}
\newcommand{\tr}{\widetilde{\rho}}
\newcommand{\tz}{\widetilde{\zeta}}
\newcommand{\tv}{\widetilde{\varphi}}
\newcommand{\tO}{\widetilde{\Omega}}
\newcommand{\hv}{\hat{\varphi}}
\newcommand{\tu}{\tilde{u}}
\newcommand{\usc}{{\sf usc}}
\newcommand{\tF}{\tilde{F}}
\newcommand{\debar}{\overline{\de}}
\newcommand{\Z}{\mathbb Z}
\newcommand{\C}{\mathbb C}
\newcommand{\Po}{\mathbb P}
\newcommand{\zbar}{\overline{z}}
\newcommand{\G}{\mathcal{G}}
\newcommand{\So}{\mathcal{U}}
\newcommand{\Ko}{\mathcal{K}}
\newcommand{\U}{\mathcal{U}}
\newcommand{\B}{\mathbb B}
\newcommand{\oB}{\overline{\mathbb B}}
\newcommand{\Cur}{\mathcal D}
\newcommand{\Dis}{\mathcal Dis}
\newcommand{\Levi}{\mathcal L}
\newcommand{\SP}{\mathcal SP}
\newcommand{\Sp}{\mathcal Q}
\newcommand{\Ma}{\mathcal M}
\newcommand{\Co}{\mathcal C}
\newcommand{\Hol}{{\sf Hol}(\mathbb H, \mathbb C)}
\newcommand{\Aut}{{\sf Aut}(\mathbb D)}
\newcommand{\D}{\mathbb D}
\newcommand{\oD}{\overline{\mathbb D}}
\newcommand{\oX}{\overline{X}}
\newcommand{\loc}{L^1_{\rm{loc}}}
\newcommand{\loci}{L^\infty_{\rm{loc}}}
\newcommand{\la}{\langle}
\newcommand{\ra}{\rangle}
\newcommand{\thh}{\tilde{h}}
\newcommand{\N}{\mathbb N}
\newcommand{\kd}{\kappa_D}
\newcommand{\Hr}{\mathbb H}
\newcommand{\ps}{{\sf Psh}}
\newcommand{\tg}{\widetilde{\gamma}}

\newcommand{\subh}{{\sf subh}}
\newcommand{\harm}{{\sf harm}}
\newcommand{\ph}{{\sf Ph}}
\newcommand{\tl}{\tilde{\lambda}}
\newcommand{\ts}{\tilde{\sigma}}

\def\v{\varphi}
\def\Re{{\sf Re}\,}
\def\Im{{\sf Im}\,}

\def\dist{{\rm dist}}
\def\const{{\rm const}}
\def\rk{{\rm rank\,}}
\def\id{{\sf id}}
\def\aut{{\sf aut}}
\def\Aut{{\sf Aut}}
\def\CR{{\rm CR}}
\def\GL{{\sf GL}}
\def\U{{\sf U}}

\def\la{\langle}
\def\ra{\rangle}

\newtheorem{theorem}{Theorem}[section]
\newtheorem{lemma}[theorem]{Lemma}
\newtheorem{proposition}[theorem]{Proposition}
\newtheorem{corollary}[theorem]{Corollary}

\theoremstyle{definition}
\newtheorem{definition}[theorem]{Definition}
\newtheorem{example}[theorem]{Example}

\theoremstyle{remark}
\newtheorem{remark}[theorem]{Remark}
\numberwithin{equation}{section}

\title[Pluricomplex Poisson kernel in strongly pseudoconvex domains]{The pluricomplex  Poisson kernel  for strongly pseudoconvex domains}
\author[F. Bracci]{Filippo Bracci}
\author[A. Saracco]{Alberto Saracco}
\author[S. Trapani]{Stefano Trapani}
\address{F. Bracci: Dipartimento Di Matematica\\
Universit\`{a} di Roma \textquotedblleft Tor Vergata\textquotedblright\ \\
Via Della Ricerca Scientifica 1, 00133 \\
Roma, Italy} \email{fbracci@mat.uniroma2.it}
\address{A. Saracco: Dipartimento di Scienze Matematiche, Fisiche e Informatiche\\ Universit\`a di Parma\\ Parco Area delle Scienze 53/A, 43124\\ Parma, Italy} \email{alberto.saracco@unipr.it}
\address{S. Trapani: Dipartimento Di Matematica\\
Universit\`{a} di Roma \textquotedblleft Tor Vergata\textquotedblright\ \\
Via Della Ricerca Scientifica 1, 00133 \\
Roma, Italy} \email{trapani@mat.uniroma2.it}


\thanks{Partially supported by PRIN {\sl Real and Complex
Manifolds: Topology, Geometry and holomorphic dynamics} n.2017JZ2SW5, by GNSAGA of INdAM and  by the MIUR Excellence Department Project awarded to the
Department of Mathematics, University of Rome Tor Vergata, CUP E83C18000100006}

\begin{abstract}
In this paper we introduce, via a Phragm\'en-Lindel\"of type theorem, a maximal plurisubharmonic function in a strongly pseudoconvex domain. We call such a function  the {\sl pluricomplex Poisson kernel} because it shares many properties with the classical Poisson kernel of the unit disc. In particular, we show that such a function is continuous, it is zero on the boundary except at one boundary point where it has a non-tangential simple pole, and reproduces pluriharmonic functions. We also use such a function to obtain a new ``intrinsic'' version of the classical Julia's Lemma and Julia-Wolff-Carath\'eodory's Theorem.
\end{abstract}

\maketitle
\tableofcontents

\section{Introduction}

The classical (negative) Poisson kernel in the unit disc $\D:=\{\zeta\in \C: |\zeta|<1\}$ with pole at $p\in\partial \D$  is defined as $P_{\D,p}(\zeta)=-\frac{1-|\zeta|^2}{|p-\zeta|^2}$. It is a harmonic function, which is zero on $\partial \D\setminus\{p\}$ and has a simple pole along non-tangential limits at $p$. The sub-level sets of $P_{\D, p}$ are the horocycles $E(p,R)$, $R>0$, with vertex in $p$, which are just discs in $\D$ tangent to $p$ (see, {\sl e.g.}, \cite[Section 1.4]{BCDbook}). When $p=1$ we simply write $P_\D$ instead of $P_{\D, 1}$. 

The classical Phragm\'en-Lindel\"of Theorem (see, {\sl e.g.} \cite[Lemma 5.2]{BPT}) states that, for every $c>0$, $cP_{\D}$ is the maximal element of the family of negative subharmonic functions $u$  in $\D$ such that $\limsup_{r\to 1}u(r)(1-r)\leq -2c$. This maximality is fundamental to proving the following version of the classical Julia-Wolff-Carath\'eodory's Theorem (see, {\sl e.g.} \cite[Section 1.4]{BCDbook} or \cite[Section 1.2]{Ababook}):
\begin{theorem}[Julia-Wolff-Carath\'eodory's Theorem]\label{JWC-intro}
Let $f:\D \to \D$ be holomorphic. Let $p\in \partial \D$ and let
\[
\lambda_p:=\inf_{q\in\partial \D}\sup_{z\in \D}\frac{P_{\D,p}(z)}{P_{\D,q}(f(z))}.
\]
If $\lambda_p<+\infty$ then there exists a unique point $q\in \partial D$ such that $f(E(p,R))\subseteq E(q,\lambda_p R)$  for all $R>0$. Moreover, $\angle\lim_{z\to p}f(z)=q$ and $\angle\lim_{z\to p}f'(z)=\overline{p}q\lambda_p$.
\end{theorem}

It is interesting to note that by the Phragm\'en-Lindel\"of theorem, the global condition $\lambda_p<+\infty$ is equivalent to the local condition $\limsup_{(0,1)\ni r\to 1}\frac{P_{\D,p}(rp)}{P_{\D,q}(f(rp))}<+\infty$.

Moreover, as it is well known, the Poisson kernel is very effective for reproducing formulas, in particular, it allows to reproduce  harmonic functions in $\D$ which extends continuously to the boundary (see, {\sl e.g.} \cite[Section 1.6]{BCDbook}). 

In higher dimension, in \cite{B-P, BPT}, the first and last named authors with G. Patrizio introduced a maximal plurisubharmonic function $\Omega_{D,p}$, in case $D\subset\C^n$ is a bounded strongly convex domains with smooth boundary and $p\in \partial D$ which solves a Phragm\'en-Lindel\"of type problem, being the maximum of all negative plurisubharmonic functions in $D$ which have at most simple poles along non-tangential limits at $p$. The function $\Omega_{D,p}$, which was constructed using the Chang-Hu-Lee spherical representation \cite{CHL}, is smooth and regular on $\overline{D}\setminus\{p\}$, zero on $\partial D\setminus\{p\}$ and solves a complex Monge-Amp\`ere type problem. Its sub-level sets correspond to Abate's horospheres \cite{Ababook, A0} (which are in fact the Buseman horospheres for the Kobayashi metric) and the associated Monge-Amp\`ere foliation is formed by the complex geodesics of $D$ whose closure contains $p$. The function $\Omega_{D,p}$ can also be used to reproduce pluri(sub)harmonic functions, as it is essentially the Demailly's \cite{De0} Poisson measure on $\partial D$. Also, in \cite[Section 2]{BCD}, it was proved that $\Omega_{D,p}$ can be used to obtain a version of Julia's Lemma and Julia-Wolff-Carath\'eodory's Theorem in strongly convex domains, relating Abate's \cite{A0} version of those theorems to the pluricomplex Poisson kernel. For those reasons, $\Omega_{D,p}$ was called the pluricomplex Poisson kernel of $D$ with pole at $p$.

In the recent paper \cite{HW}, via a careful study of properties of complex geodesics in lower regular convex domains, X. Huang and X. Wang showed that a pluricomplex Poisson kernel is well defined, with essentially the same properties, in bounded strongly linearly convex domains with $C^3$-smooth boundary.

Also, in \cite{Po}, E. Poletsky used the pluricomplex Poisson kernel introduced in \cite{B-P, BPT} to show that the pluripotential boundary of a bounded smooth strongly convex domain is homeomorphic to its Euclidean boundary. 

In this paper we deal with the pluricomplex Poisson kernel for bounded strongly pseudoconvex domains with smooth boundary. Using the techniques introduced in \cite{HW} one can lower the required boundary regularity, but we are not interested in this aspect in this paper.

 Let $D\subset \C^n$ be a strongly pseudoconvex domain with smooth boundary. Part of the results are in fact proved for strongly pseudoconvex domains in Stein manifolds,  but for the sake of simplicity, in this introduction we restrict to the case of $\C^n$.  

Let $p\in\partial D$ and let $\nu_p$ be the outer unit normal of $\partial D$ at $p$. We consider the following family $\mathcal S_p$:
\begin{equation*}
\begin{cases}
u \in \ps (D) \\
u<0 \quad \hbox{in $D$}\\
\displaystyle{\limsup_{t\to 1} u(\gamma(t))(1-t)\leq
-2\Re\frac{1}{\langle \gamma'(1), \nu_p\rangle}},
\end{cases}
\end{equation*}
where $\gamma:[0,1]\to D\cup\{p\}$ is any smooth curve such that $\gamma([0,1))\subset D$, $\gamma(1)=p$,  and $\gamma'(1)\not\in T_p\partial D$, $\ps(D)$ denotes the family of plurisubharmonic functions in $D$ and $\langle \cdot, \cdot \rangle$ denotes the standard Hermitian product in $\C^n$.

Our first result is the following:

\begin{theorem}\label{Thm:intro1}
Let $D\subset \C^n$ be a strongly pseudoconvex domain with smooth boundary and let $p\in \partial D$. Then, there exists a maximal plurisubharmonic function $\Omega_{D,p}\in \mathcal S_p$, which we call the {\sl pluricomplex Poisson kernel} of $D$ at $p$, such that $u\leq \Omega_{D,p}$ for all $u\in \mathcal S_p$. Moreover, $\Omega_{D,p}$ is continuous in $\overline{D}\setminus\{p\}$, $\Omega_{D,p}(x)=0$ for all $x\in \partial D\setminus\{p\}$ and
\[
\lim_{t\to 1} \Omega_{D,p}(\gamma(t))(1-t)=
  -2\Re\frac{1}{\langle \gamma'(1), \nu_p\rangle},
\]
for any $\gamma:[0,1]\to D\cup\{p\}$ which is a smooth curve such that $\gamma([0,1))\subset D$, $\gamma(1)=p$ and $\gamma'(1)\not\in T_p\partial D$.
\end{theorem}

This result holds also when $D$ is a smooth bounded strongly pseudoconvex domain in a Stein manifold  and $\langle \cdot, \cdot \rangle$ is replaced with a linear functional at $p$ which defines $T_p\partial D$ (see Section~\ref{one} for details).

In strongly convex domains the continuity of $\Omega_{D,p}$ and the behavior on $\partial D\setminus\{p\}$ come for free by the Chang-Hu-Lee spherical representation. In strongly pseudoconvex domains, this is not the case. In Proposition~\ref{prop-Omega} we prove the first properties of $\Omega_{D,p}$, while  continuity is proved in Section~\ref{continuity} using Jensen's measures.

In particular, by the previous theorem it follows that $\Omega_{D,p}$ is a solution of the following complex Monge-Amp\`ere problem:
\begin{equation*}
\begin{cases}
u\in \ps(D)\cap L^\infty_{loc}(D)\\
(dd^c u)^n=0 \quad \hbox{in $D$}\\
u<0 \quad \hbox{in $D$}\\
u(x)=0 \quad \hbox{for $x\in\partial D$}\\
u(z)\sim |z-p|^{-1} \quad \hbox{as $z\to p$ non-tangentially}.
\end{cases}
\end{equation*}

It is not known whether $\Omega_{D,p}$ is the {\sl unique} solution to such a problem, not even in the convex case. However, as in the convex case, we can prove that if $u$ is a maximal negative plurisubharmonic function in $D$ such that $\lim_{z\to x}u(z)=0$ for all $x\in\partial D$ and $\lim_{z\to p}\frac{\Omega_{D,p}(z)}{u(z)}=1$, then $u\equiv \Omega_{D,p}$ (see Proposition~\ref{Prop:unique}). 

Next, we study in depth the behavior of $\Omega_{D,p}$ near $p$. To this aim, we introduce a tool, which we call ``entrapping strongly pseudoconvex domains between strongly convex domains'', which might be useful for other purposes (see Section~\ref{Trap}). Using this, we are able to compare $\Omega_{D,p}$ with the pluricomplex Poisson kernels of strongly convex domains. More precisely, if $U$ is a small neighborhood of $p$ so that $B:=U\cap D$ is biholomorphic to a strongly convex domain with smooth boundary, then we show that $\lim_{z\to p}\frac{\Omega_{D,p}(z)}{\Omega_{B,p}(z)}=1$, and that the two kernels in fact coincide on the ``quasi'' complex tangential directions at $p$. From this, making use of Lempert's theory \cite{Le1, Le2, Le3} and Huang's preservation principle \cite{H},  we prove that there exists an open set $J\subset D$, whose closure contains $p$, and contains all ``quasi complex-tangential directions'' at $p$, such that $\Omega_{D,p}$ is $C^\infty$-smooth on $J$, $(dd^c \Omega_{D,p})^{n-1}\neq 0$ on $J$ and the associated Monge-Amp\`ere foliation on $J$ formed by complex geodesics of $D$ whose closure contains $p$, which are also holomorphic retracts of $D$ (see Proposition~\ref{Prop:lle}).

Then we extend  Julia's Lemma (see Theorem~\ref{Julia}) and Julia-Wolff-Carath\'eodory's Theorem (see Theorem~\ref{Thm:JWC}) to strongly pseudoconvex domains using the pluricomplex Poisson kernels. Versions of these theorems have been proved in strongly pseudoconvex domains by M. Abate \cite{A1}. The novelty of our result is that, as in Theorem~\ref{JWC-intro} and in strongly convex domains, we can  relate the number $\lambda_p$ with the behavior of the normal part of the derivative of $f$ along the normal direction, a link which was missing in Abate's result. We also show that the hypotheses of Abate's theorem are equivalent to the one using pluricomplex Poisson's kernels.

The final aim of the paper is to prove a representation formula for pluri(sub)harmonic functions on strongly pseudoconvex domains. We first show (see Proposition~\ref{Prop:uguali}) that, in case the pluricomplex Green function $G_D$ of $D$ is symmetric, then for all $z\in D$,
\begin{equation}\label{eq-intro}
 - \frac{\partial G_D(z,p)}{\partial \nu_p}=\Omega_{D,p}(z).
\end{equation}
In order to prove this formula, we use the symmetry of $G_D$ and a result of Z. B\l ocki \cite{Bl, Bl2} to show that $ - \frac{\partial G_D(\cdot,p)}{\partial \nu_p}$ is a maximal plurisubharmonic function in $D$, zero on $\partial D\setminus\{p\}$. Then, using the ``entrapping trick'' we show that $\lim_{z\to p}\frac{- \frac{\partial G_D(z,p)}{\partial \nu_p}}{\Omega_{D,p}(z)}=1$, from which it follows that the two functions are equal.

It should be noticed that, arguing as in \cite{Po}, the previous formula allows to prove that the Poletsky potential boundary of $D$ is homeomorphic to the Euclidean boundary of $D$.

Once we have \eqref{eq-intro} at hand, we can prove that  Demailly's Poisson measure on $\partial D$ concides with $|\Omega_{D,p}(z)|^n \omega_{\partial D}$, where $\omega_{\partial D}$ is a measure on $\partial D$ steaming essentially from the Levi form of $D$ (see Lemma~\ref{Lem:Green-Demailly}). Hence (see Theorem~\ref{Trepfor}), we get that  if $f$ is plurisubharmonic in $D$ and continuous on $\overline{D}$, 
\begin{equation*}
\begin{split}
f(z) &= \frac{1}{(2\pi)^n}\int_{\partial D} f(\xi) |\Omega_{D,\xi}(z)|^n\omega_{\partial D}(\xi)\\&-\frac{1}{(2\pi)^n}\int_{w\in D} |G_D(z, w)| dd^c f(w)\wedge (dd^c G_D(z,w))^{n-1}.
\end{split}
\end{equation*} 

\medskip

We warmly thank the referees for the careful reading of the original manuscript and for many comments which improved a lot the paper.

\section{Definition and first properties}\label{one}

Let $M$ be a  Stein manifold with complex structure $J$.
Let $D\subset\subset M$ be a strongly pseudoconvex domain with
smooth boundary. Given $p\in \de D$ we let $T_p^{\mathbb C} \de
D$ be the complex tangent space to $\de D$ at p. We write
\[
T_p \de D\setminus T_p^{\mathbb C} \de D= K^+\cup K^-,
\]
where $v \in K^+$ if $-Jv$ points outside the domain $D$. Let
$\alpha_p: T_p\de D\to \R$ be a linear map such that ${\sf Ker}
\alpha_p = T_p^{\mathbb C} \de D$ and $\alpha_p|_{K^+}>0$, in
other words, $\alpha_p$ is a point of the fiber over $p$ of the
bundle of contact forms over $\de D$. Notice that $\alpha_p$ is
uniquely defined up to multiplication by a positive constant.
In the sequel we will denote by $\mathcal H_p(\de D)$ the set
of forms $\alpha_p$ as above.

Let $v_0\in K^+$ be such that $\alpha_p(v_0)=1$. Given $w\in
T_p M$ we can write uniquely $w=u - \theta_{v_0}(w) J v_0$ for
some $u\in T^\C_p\de D$ and  $\theta_{v_0}(w)\in \C$.

Note that if $v_1\in T_p \partial D$ is another  vector such that
$\alpha_p(v_1)=1$, then $v_1-v_0\in T^\C_p\de D$ and
$\theta_{v_1}(w)=\theta_{v_0}(w)$ for every $w\in T_p M$. Thus $\theta_{v_0}$
depends only on $\alpha_p$ and we can write
$\theta_{\alpha_p}:=\theta_{v_0}$. Note that
$\theta_{\alpha_p}: T_p M\to \C$ is $\C$-linear, $T_p\de D=
\ker \Re (\theta_{\alpha_p})$ and $\Im (\theta_{\alpha_p}|_{T_p
\de D})=\alpha_p$.

\begin{definition}
For short, in the rest of the paper, we call the couple $(\alpha_p, \theta_{\alpha_p})$ a {\sl defining couple} for $T_p^\C\partial D$.
\end{definition}

We let $\Gamma_p$ be the set of all $C^\infty$ curves
$\gamma:[0,1]\to D\cup\{p\}$ such that $\gamma(t)\in D$ for $t\in [0,1)$, $\gamma(1)=p$ and
$\gamma'(1)\not\in T_p\de D$. Note that, for what we discussed
above, $\gamma'(1)\not\in T_p\de D$ if and only if $\Re
\theta_{\alpha_p}(\gamma'(1))> 0$.

Consider the following family $\mathcal S_{\alpha_p}(D)$:
\begin{equation}\label{family1}
\begin{cases}
u \in \ps (D) \\
u<0 \quad \hbox{in $D$}\\
\displaystyle{\limsup_{t\to 1} u(\gamma(t))(1-t)\leq
-2\Re[\theta_{\alpha_p}(\gamma'(1))^{-1}]} \quad \hbox{for all
}\gamma\in \Gamma_p,
\end{cases}
\end{equation}

\begin{remark}\label{Poissondisco}
Let $\D:=\{\zeta\in \C: |\zeta|<1\}$. Let
$P_\D(\zeta):=-(1-|\zeta|^2)/(|1-\zeta|^2)$ be the Poisson
kernel. If $\sigma:[0,1]\to \oD$ is a $C^1$-curve such that
$\sigma(1)=1$ and $\sigma([0,1))\subset\D$ with $\sigma'(1)\neq
0$ then
\[
\lim_{t\to 1^-} P_\D(\sigma(t))(1-t)=-2\Re
\frac{1}{\sigma'(1)}.
\]
\end{remark}

\begin{lemma}\label{esiste limitato}
Let $M$ be a Stein manifold and let $D\subset\subset M$ be a
strongly pseudoconvex domain with smooth boundary. Let $p, q\in \de D$,
$q\neq p$. For every $\alpha_p\in \mathcal H_p(\de D)$, the
family $\mathcal S_{\alpha_p}(D)$ contains a function $u_q:D\to
\R$ with the following properties:
\begin{enumerate}
  \item $u_q$ is continuous in $D$ and extends continuously on $\de
D\setminus \{p\}$,
  \item  $u_q(q)=0$,
  \item $\lim_{t\to 1} u_q(\gamma(t))(1-t)=-2 \Re
  [\theta_{\alpha_p}(\gamma'(1))]^{-1}$ for all $\gamma\in
  \Gamma_p$.
\end{enumerate}
\end{lemma}

\begin{proof}
Replacing $M$ with a suitable Stein neighborhood of $\overline{D}$, we can assume that $\overline{D}$ is holomorphically convex in $M$. 

We can choose a positive definite Hermitian metric $\la \cdot,
\cdot \rangle$ on $T_p M$ and $v_0\in K^+$  such that  $\nu_p:=- J v_0$ is a unit
normal vector pointing outward. Thus for every $w\in T_p M$ we
have $\theta_{\alpha_p}(w)=\la w, \nu_p\ra$.

Let $f:\overline{D}\to \C$ be a holomorphic peak function for
$D$ at $p$, namely, $f$ is a smooth function on $\overline{D}$
such that $f:D\to \D$ is holomorphic, $f(p)=1$ and
$f(\overline{D}\setminus\{p\})\subset \D$ (in case $M\neq \C^n$, it exists because  $\overline{D}$  is holomorphically convex in $M$, see
\cite[IX.C.7]{GR} \cite[Corollary 11]{Fo}).

By Hopf's lemma, $\lambda:=df_p(\nu_p)=\frac{\de f}{\de
\nu_p}(p)> 0$. Also, $|f|^2$ restricted to $\de D$ has a
maximum at $p$ and $f(p)=1$. Thus $d(|f|^2)_p=2\Re (f(p)
df_p)=2 \Re (df_p)$ vanishes on $T_p \de D$. Hence $df_p$ is a
positive multiple of $\theta_{\alpha_p}$, {\sl i.e.},
$df_p(w)=\lambda \la w,\nu_p\ra$ for all $w\in T_pM$.

Now, let $\tu(z):=\lambda P_{\D,1}\circ f$. Let $\gamma\in
\Gamma_p$. Then by Remark \ref{Poissondisco}
\begin{equation*}
\lim_{t\to 1} \tu(\gamma(t))(1-t) = - 2\lambda \Re (df_p
(\gamma'(1)))^{-1}) = -\frac{2\lambda}{\lambda} \Re
\frac{1}{\la \gamma'(1), \nu_p\ra} = -2\Re \frac{1}{\la
\gamma'(1), \nu_p\ra}.
\end{equation*}
Now let $g:\overline{D}\to \C$ be a holomorphic peak function
for $D$ at $q$. Let $\phi(z):=|g(z)|^2-1$. Then $\phi\in
\ps(D)$, $\phi$ is continuous up to $\overline{D}$, $\phi(q)=0$
and $\phi(z)<0$ for $z\in \overline{D}\setminus\{p\}$. Let $U$
be a relatively compact open neighborhood of $p$ whose closure
does not contain $q$. Then  $\phi|_{\overline{U}\cap
\overline{D}}<0$. Since both $\tu$ and $\phi$ are continuous on
$\de U \cap \overline{D}$, there exists $\rho>0$ such that
$\rho \phi < \tu$ on $\de U \cap \overline{D}$.  Therefore the
function
\[
u_q(z):=
\begin{cases}
\tu(z) & \hbox{if}\ z\in U\\
\max\{\tu(z), \rho \phi(z)\} & \hbox{if}\ z\in D\setminus U
\end{cases}
\]
belongs to $\mathcal S_{\alpha_p}(D)$ and has the properties
stated in the lemma.
\end{proof}

\begin{definition}
Let $M$ be a Stein manifold. Let $D\subset\subset M$ be a
strongly pseudoconvex domain with smooth boundary. Let $p\in \de D$ and
let $\alpha_p\in \mathcal H_p(\de D)$. The {\sl pluricomplex
Poisson kernel of $D$ at $p$ relative to $\alpha_p$} is
\[
\Omega^{\alpha_p}_{D,p}(z):=\sup \{u(z): u\in \mathcal
S_{\alpha_p}(D)\}.
\]
\end{definition}

\begin{remark}\label{pois-conv}
The pluricomplex Poisson kernel defined in \cite{B-P} and
\cite{BPT} for a strongly convex domain $D\subset \C^n$ is the
one corresponding to the standard form $\kappa_p$ such that
$\kappa_p(J\nu_p)=1$ and ${\sf Ker} \kappa_p =T_p^\C(\de D)$
where $\nu_p$ is the outer unit normal vector to $\de D$ at $p$ with
respect to the standard Hermitian metric of $\C^n$.
\end{remark}

\begin{proposition}\label{ovvia}
Let $M$ be a Stein manifold. Let $D\subset\subset M$ be a
strongly pseudoconvex domain with smooth boundary. Let $p\in \de D$ and
let $\alpha_p\in \mathcal H_p(\de D)$. Then
\begin{enumerate}
  \item for all $\rho>0$, $u\in \mathcal S_{\rho \alpha_p}(D)$
  if and only if $\rho u\in \mathcal S_{\alpha_p}(D)$. In
  particular, $\Omega^{\alpha_p}_{D,p}=\rho \Omega^{\rho
  \alpha_p}_{D,p}$.
  \item Let $M'$ be another complex manifold and let
  $D'\subset\subset M'$ be a strongly pseudoconvex domain with
smooth boundary. Let $q\in \de D'$. Let $F:\overline{D}\to
\overline{D'}$ be a diffeomorphism such that $F:D\to D'$ is
holomorphic. Then $u \in \mathcal S_{\alpha_q}(D')$ if and only
if $F^\ast(u):=u\circ F\in \mathcal S_{F^\ast(\alpha_q)}(D)$.
In particular
$F^\ast(\Omega^{\alpha_q}_{D',q})=\Omega^{F^\ast(\alpha_q)}_{D,F^{-1}(q)}$.
\end{enumerate}
\end{proposition}

\begin{proof}
These are just direct computations. We only point out that, in
(2), $dF_p(K^+)=K^+$ for all $p\in \de D$, so that if
$\alpha_q\in \mathcal H_q(\de D')$ then $F^\ast(\alpha_q)\in
\mathcal H_{F^{-1}(q)}(\de D)$.
\end{proof}

In the sequel we will need to characterize the pluricomplex
Poisson kernel as the supremum of another (equivalent) family.
Let $M$ be a complex manifold. Let $D\subset\subset M$ be a
strongly pseudoconvex domain with smooth boundary. Let $p\in
\de D$ and let $\alpha_p\in \mathcal H_p(\de D)$. Let
$D'\subset D$ be a strongly pseudoconvex domain with smooth
boundary which is tangent to $D$ at $p$. Let
\begin{equation*}
\So_{D',\alpha_p}(D):=\{u\in \ps(D): u<0, u|_{D'}\leq
\Omega^{\alpha_p}_{D',p} \}.
\end{equation*}

\begin{lemma}\label{palla-equiv}
Let $M$ be a Stein manifold. Let $D\subset\subset M$ be a
strongly pseudoconvex domain with smooth boundary. Let $p\in \de D$ and
let $\alpha_p\in \mathcal H_p(\de D)$. Let $D'\subset D$ be a
domain with smooth boundary which is biholomorphic to a bounded
strongly convex domain of $\C^n$ with smooth boundary, and
assume that $D'$ is tangent to $D$ at $p$. Then $\mathcal
S_{\alpha_p}(D)=\So_{D',\alpha_p}(D)$.
\end{lemma}

\begin{proof}
Let $u\in \So_{D',\alpha_p}(D)$. If $\gamma\in \Gamma_p$, then
eventually the curve $\gamma$ is contained in $D'$. Let
$F:D'\to D''$ be a biholomorphism from $D'$ to a bounded
strongly convex domain $D''\subset \C^n$. By Fefferman's theorem
\cite{Fe} the map $F$ extends smoothly to $\overline{D'}$. By
Proposition \ref{ovvia}
\[
\Omega^{\alpha_p}_{D',p}=\rho^{-1} F^\ast(\Omega_{D'', F(p)}),
\]
where $\Omega_{D'',F(p)}$ is the pluricomplex Poisson kernel of
$D''$ at $F(p)$ relative to the standard form $\kappa_{F(p)}$
(see Remark \ref{pois-conv}) and $\rho>0$ is such that
$\alpha_p=\rho F^\ast(\kappa_{F(p)})$. By \cite[Corollary
5.3]{BPT} for all $\tg\in \Gamma_{F(p)}$ it holds
\begin{equation}\label{convo}
\lim_{t\to 1}\rho^{-1} \Omega_{D'',
F(p)}(\tg(t))(1-t)=-2\Re(\la \tg'(1),\nu_{F(p)}\ra^{-1}),
\end{equation}
where $\nu_{F(p)}$ is the standard outer unit normal vector to $\de
D''$ at $F(p)$. Hence
\begin{equation*}
\begin{split}
\limsup_{t\to 1}&\,u(\gamma(t))(1-t)\leq \limsup_{t\to
1}\Omega^{\alpha_p}_{D',p}(\gamma(t))(1-t)=\\ & \limsup_{t\to
1}\rho^{-1} \Omega_{D'', F(p)}(F\circ \gamma(t))(1-t)=
-\frac{2}{\rho}\Re[\la dF_p(\gamma'(1)),\nu_{F(p)}\ra^{-1}]=\\
&
-\frac{2}{\rho}\Re[\theta_{\kappa_{F(p)}}(dF_p(\gamma'(1))^{-1}]
=-2\Re[\theta_{\alpha_p}(\gamma'(1))^{-1}].
\end{split}
\end{equation*}
Thus $u\in \mathcal S_{\alpha_p}(D)$.

Conversely, if $u\in \mathcal S_{\alpha_p}(D)$ then clearly
$u|_{D'}\in \mathcal S_{\alpha_p}(D')$. But then $u|_{D'}\leq
\Omega^{\alpha_p}_{D',p}$ by  \cite[Theorem 5.1]{BPT}. Hence $u
\in \So_{D',\alpha_p}(D)$.
\end{proof}

As a consequence  we have the following proposition:

\begin{proposition}\label{prop-Omega}
Let $M$ be a Stein manifold. Let $D\subset\subset M$ be a
strongly pseudoconvex domain with smooth boundary. Let $p\in \de D$ and
let $\alpha_p\in \mathcal H_p(\de D)$. Then
\begin{enumerate}
  \item $\Omega^{\alpha_p}_{D,p}$ is upper semicontinuous and belongs to $\mathcal S_{\alpha_p}(D)$,
  \item $\lim_{z\to q}\Omega^{\alpha_p}_{D,p}(z)=0$ for all
  $q\in \de D\setminus\{p\}$,
  \item $\lim_{t\to 1} \Omega^{\alpha_p}_{D,p}(\gamma(t))(1-t)=-2 \Re
  [\theta_{\alpha_p}(\gamma'(1))]^{-1}$ for all $\gamma\in
  \Gamma_p$,
  \item $\Omega^{\alpha_p}_{D,p}\in {L}^\infty_{\sf{loc}}(D)$,
  \item $\Omega^{\alpha_p}_{D,p}$ is a maximal plurisubharmonic function in $D$, hence
  $(dd^c \Omega^{\alpha_p}_{D,p})^n\equiv 0$ in $D$.
\end{enumerate}
\end{proposition}

\begin{proof}
(1) Let $v$ be the upper semicontinuous regularization of
$\Omega^{\alpha_p}_{D,p}$. Let $B\subset D$ be a domain
biholomorphic to a ball in $\C^n$ which is tangent to $\de D$
at $p$.  By Lemma \ref{palla-equiv},
$\Omega^{\alpha_p}_{D,p}|_{B}\leq \Omega^{\alpha_p}_{B,p}$.
Since $\Omega^{\alpha_p}_{B,p}$ is smooth, $v|_B\leq
\Omega^{\alpha_p}_{B,p}$ as well. Hence $v\in
\So_{B,\alpha_p}(D)$, hence $v\leq \Omega^{\alpha_p}_{D,p}$,
again by Lemma \ref{palla-equiv}. This proves that
$\Omega^{\alpha_p}_{D,p}$ is upper semicontinuous, hence
plurisubharmonic, and then belongs to $\mathcal
S_{\alpha_p}(D)$.

(2), (3) and (4) Let $q\in \de D\setminus\{p\}$. Let $u_q\in
\mathcal S_{\alpha_p}(D)$ be given by Lemma \ref{esiste
limitato}. Then
\[
u_q\leq \Omega^{\alpha_p}_{D,p}\leq 0,
\]
which proves that $\Omega^{\alpha_p}_{D,p}$ is locally bounded
in $D$. Moreover, since $u_q(z)\to 0$ as $z\to q$, $\lim_{z\to
q}\Omega^{\alpha_p}_{D,p}(z)=0$. Also, (3) follows from Lemma
\ref{esiste limitato}.(3) and $\Omega^{\alpha_p}_{D,p}\in
\mathcal S_{\alpha_p}(D)$.

(5) Let $U$ be a relatively compact open subset of $D$. Let $u$
be an upper semicontinuous function on $\overline{U}$ which is
plurisubharmonic in $U$ and such that $u\leq
\Omega^{\alpha_p}_{D,p}$ on $\de U$. By the maximum principle,
$u<0$ in $U$. Let
\[
\tu(z):=\begin{cases}
 \max\{u(z), \Omega^{\alpha_p}_{D,p}(z)\} & \hbox{if}\
z\in U\\
\Omega^{\alpha_p}_{D,p}(z) & \hbox{if}\ z\in
D\setminus U
\end{cases}
\]
Then $\tu\in \mathcal S_{\alpha_p}(D)$. Hence $\tu\leq
\Omega^{\alpha_p}_{D,p}$. Therefore $u\leq
\Omega^{\alpha_p}_{D,p}$ in $U$, showing that
$\Omega^{\alpha_p}_{D,p}$ is maximal. By \cite{B-T1},
\cite{B-T2} it follows $(dd^c \Omega^{\alpha_p}_{D,p})^n\equiv
0$.
\end{proof}

Lemma \ref{palla-equiv} holds in general for any couple of
strongly pseudoconvex domains:

\begin{proposition}\label{sotto-colle}
Let $M$ be a Stein manifold. Let $D\subset\subset M$ be a
strongly pseudoconvex domain. Let $p\in \de D$ and
let $\alpha_p\in \mathcal H_p(\de D)$. Let $D'\subset D$ be a
strongly pseudoconvex domain with smooth boundary tangent to
$D$ at $p$. Then $\mathcal
S_{\alpha_p}(D)=\So_{D',\alpha_p}(D)$.
\end{proposition}

\begin{proof}
The proof goes exactly as in Lemma \ref{palla-equiv}, replacing
\eqref{convo} with  Proposition \ref{prop-Omega}.(3).
\end{proof}

\section{Continuity}\label{continuity}

\begin{lemma}\label{approssima}
Let $M$ be a Stein manifold. Let $D\subset\subset M$ be a
strongly pseudoconvex domain with smooth boundary. Let $u$ be upper
semicontinuous on $\overline{D}$ such that $u \in \ps(D)$. Then
there exists a sequence $\{u_j\}$ of continuous   functions on
$\overline{D}$ with range in $(-\infty,+\infty)$ such that
$u_j\in \ps(D)$ and $u_j\searrow u$ pointwise on $\overline{D}$.
\end{lemma}

\begin{proof}
By Forn{\ae}ss embedding theorem \cite[Theorem 9]{Fo} there
exists a holomorphic embedding $G$ of $\overline{D}$ into a
strongly convex domain $\overline{C}$ such that $G(D)\subset
C$. Hence $G(D)$ is an analytic variety in the $B$-regular
domain $C$, and we can apply Wikstr\"om's theorem \cite[Theorem
2.3]{Wk}. Therefore we can find a sequence $\tilde{u}_j$ of
continuous functions on $G(\overline{D})$ which are
plurisubharmonic on $G(D)$ and $\tilde{u}_j\searrow
(G^{-1})^\ast u$ pointwise on $G(\overline{D})$. Setting
$u_j:=\max\{G^{\ast} \tilde{u}_j,-j\}$ we get the claim.
\end{proof}

\begin{proposition}\label{continuitygenerale}
Let $M$ be a Stein manifold of dimension $n$. Let
$D\subset\subset M$ be a strongly pseudoconvex domain with
smooth boundary. Let $p\in \de D$. Let $p\in \de D$ and let $\alpha_p\in
\mathcal H_p(\de D)$. Then $\Omega^{\alpha_p}_{D,p}$ is
continuous on $\overline{D}\setminus \{p\}$.
\end{proposition}

\begin{proof}
By Proposition \ref{prop-Omega}, $\Omega^{\alpha_p}_{D,p}$
belongs to the family $\mathcal S_{\alpha_p}(D)$, hence it is
upper semicontinuous.

In order to show lower-semicontinuity, we use a variation of
the method of {\sl Jensen measures} and Edwards' theorem (see
\cite{Wk}, \cite{Wk2}).

Let $a\in D$. We shall prove that $\Omega^{\alpha_p}_{D,p}$ is
continuous in $a$.

Since $M$ is Stein, there exists a holomorphic embedding $G:M\to \C^N$, for some $N\geq n$.  Let $T:=dG_p(T_pM)$ and choose a complementary $(N-n)$-dimensional complex space $L\subset \C^N$ such that the projection $\pi$ of $\C^N$ onto $T$ along $L$ does not map $G(a)$ in $dG_p(T_p(\partial D))$. By construction, there exists an open neighborhood $U$ of $p$ in $M$ such that $F:=\pi\circ G$ is a biholomorphism on $U$. Therefore, $F(\partial D)$ is strongly pseudoconvex at $F(p)$ and there is a ball $B$ in $F(U)$ that is tangent to $F(\partial D)$ at $F(p)$ and does not intersect $dF_p(T_p\partial D)$. 

We can choose coordinates $w=(w_1,\ldots, w_n)$ on $T$ (so that $T$ is naturally identified with $\C^n$) such that $F(p)=e_1=(1,0\ldots, 0)$, $dF_p(T_p\partial D)=\{w:  \Re(w_1-1)=0\}$,  and $B=\B^n$.

Let $\B':=(F|_U)^{-1}(\B^n)$. Let $h(z):=F_1(z)-1$ (here $F_1:=w_1\circ F$) and let $H:=\{z\in M: h(z)=0\}$.  Note that
$H$ is an analytic hypersurface in $M$, that $D\setminus H$ is
dense in $D$,  $a\not\in H$, $\B'\cap H=\emptyset$ and that $H$ is non singular at
$p$ and tangent to $\de D$ at $p$.

The Poisson kernel of $\B^n$ at
$e_1$ associated to the standard Hermitian metric of $\C^n$ is
given by
\[
\Omega_{\B^n,e_1}(w)=-\frac{1-\|w\|^2}{|1-w_1|^2}.
\]
Multiplying $\Omega_{\B^n,e_1}$ by a positive constant if
necessary, we may assume that $F^\ast(\Omega_{\B^n,
e_1})=\Omega^{\alpha_p}_{\B', p}$.

Let $\mathcal P^{ba}(D)$ be the cone of all bounded from above plurisubharmonic functions in $D$ and let $\mathcal P^c(\overline{D})$ be the cone of all coontinuous function in $\overline{D}$ which are plurisubharmonic in $D$. We let $\mathcal F:=|h|^2\mathcal P^{ba}(D)$ and $\mathcal F^c:=|h|^2\mathcal P^c(\overline{D})$. 

Since $|h|^2$ is bounded on $\overline{D}$, the functions in $\mathcal F$ and $\mathcal F^c$ are bounded from above. Therefore, the functions in $\mathcal F$ can be extended on $\partial D$ as upper semicontinuous functions.

 Let
\[
\tilde{g}(z)=
\begin{cases}
\Omega^{\alpha_p}_{\B', p}(z) & z\in \B'\\
0 & z\not\in \B'
\end{cases}
\]
and set $g(z):=\tilde{g}(z)|h(z)|^2$.

Note that for $z\in \B'$,
\begin{equation*}
\begin{split}
g(z)&=\Omega_{\B',p}(z)|
F_1(z)-1|^2=(\Omega_{\B^n,e_1}(w)\cdot  |w_1-1|^2)\circ
F(z)\\&=-(1-\|w\|^2)\circ F(z),
\end{split}
\end{equation*}
and then the function $g$ is $\leq 0$ and continuous on
$\overline{D}$.

Let
\[
S^{\mathcal F}(g)(z):=\sup_{u\in {\mathcal F}, u\leq g} u(z),
\]
and similarly we define $S^{\mathcal F^c}(g)(z)$.

For a point $z\in \overline{D}$ we let
\[
\mathcal I^z_{\mathcal F}:=\left\{\nu_z \ \hbox{positive finite
Borel measure on}\ \overline{D}: u(z)\leq \int_{\overline{D}} u
d\nu_z \quad\forall u\in \mathcal F\right\}
\]
and similarly define
\[
\mathcal I^z_{\mathcal F^c}:=\left\{\nu_z \ \hbox{positive finite
Borel measure on}\ \overline{D}: u(z)\leq \int_{\overline{D}} u
d\nu_z \quad\forall u\in \mathcal F^c\right\}.
\]

Since $\mathcal F^c\subseteq \mathcal F$, it follows
immediately $\mathcal I_{\mathcal F}^z\subseteq \mathcal
I_{\mathcal F^c}^z$.

We claim that actually $\mathcal I_{\mathcal F}^z= \mathcal
I_{\mathcal F^c}^z$. In order to prove such equality, let
$\nu_z\in \mathcal I^z_{\mathcal F^c}$ and let $v\in \mathcal
F$. Let $\tilde{v}$ be the upper semicontinuous extension to
$\overline{D}$ of $\overline{D}\setminus H \ni z\mapsto
v(z)/|h(z)|^2$. By definition  $\tilde{v}\in \ps(D)\cap
\usc(\overline{D})$. By Lemma \ref{approssima} there exists a
sequence $\{\tilde{u}_j\}\subset \ps(D)$ of continuous functions 
on $\overline{D}$  such that $\tilde{u}_j\searrow \tilde{v}$
pointwise on $\overline{D}$.   Set $u_j(z):=\tilde{u}_j(z)
|h(z)|^2$ for $z\in \overline{D}$. Since $\tilde{u}_j$ is 
bounded from below, then $u_j|_H\equiv 0$. By construction $u_j \in
\mathcal{F}^c$ and $u_j\searrow v$ pointwise on $\overline{D}$.
Hence
\[
\int_{\overline{D}} v d\nu_z =\lim_{j\to \infty}
\int_{\overline{D}} u_j d\nu_z\geq \lim_{j\to \infty} u_j(z)
\geq v(z),
\]
and therefore $\nu_z\in \mathcal I^z_{\mathcal F}$, proving the
equality.

Now, we let
\[
I^{\mathcal F}(g)(z):=\inf \left\{\int_{\overline{D}} g d\nu_z :
\nu_z\in {\mathcal I}^z_{\mathcal F}\right\},
\]
and similarly we define
\[
I^{\mathcal F^c}(g)(z):=\inf \left\{\int_{\overline{D}} g d\nu_z :
\nu_z\in {\mathcal I}^z_{\mathcal F^c}\right\}.
\]
Since $\mathcal I_{\mathcal F}^z= \mathcal I_{\mathcal F^c}^z$
it follows $I^{\mathcal F}(g)=I^{\mathcal F^c}(g)$.

Now, both $\mathcal F$ and $\mathcal F^c$ are cones of upper
bounded, upper semicontinuous functions, thus by Edwards'
theorem (see \cite[Theorem 2.1]{Wk2} and its proof at p. 183)
$S^{\mathcal F}(g)=I^{\mathcal F}(g)$ and $S^{\mathcal
F^c}(g)=I^{\mathcal F^c}(g)$.

Hence $S^{\mathcal F^c}(g)=S^{\mathcal F}(g)$ which implies
that the function $S^{\mathcal F}(g)$ is the supremum of
continuous functions (with, a priori, range in $[-\infty, 0])$)
and therefore it is lower semicontinuous.

Let $u\in \mathcal F$ be such that $u\leq g$. Write $u=|h|^2\tu$ for some $\tu\in \mathcal P^{ba}(D)$. Then $\tilde u\leq |h|^{-2} g$ on $D$. Since $H\cap \B'=\emptyset$, and $g\equiv 0$ outside $\B'$, it follows that $\tu \leq \Omega^{\alpha_p}_{\B',p}$ in $\B'$ and  $\tu\leq 0$ in $D$ and, actually, by the maximum principle, $\tu<0$ in $D$. In other words, $\tu\in
\So_{\B',\alpha_p}(D)$. Conversely,  if $\tu \in \So_{\B',\alpha_p}(D)$ then $u:=|h|^2\tu \in \mathcal F$ and $u\leq g$ in $D$. Therefore, 
\[
|h|^2\So_{\B',\alpha_p}(D)=\{u \in \mathcal F : u\leq g\}.
\]
By Lemma \ref{palla-equiv}, we have
\[
S^{\mathcal F^c}(g)=S^{\mathcal F}(g)=\sup_{u\in \mathcal F, u\leq g}u=|h|^{-2}\sup_{\tu \in \So_{\B',\alpha_p}(D)} \tu=|h|^{-2}\sup_{\tu \in \mathcal S_{\alpha_p}(D)} \tu=|h|^{-2}\Omega^{\alpha_p}_{D,p}.
\]
Since $h(a)\neq 0$, the previous equality implies that $\Omega^{\alpha_p}_{D,p}$ is lower semicontinuous at $a$. But $\Omega^{\alpha_p}_{D,p}$ is
upper semicontinuous at $a$ by Proposition \ref{prop-Omega},
hence it is continuous at $a$. 

By the arbitrariness of $a\in
D$, $\Omega^{\alpha_p}_{D,p}$ is continuous in $D$ and, by Proposition \ref{prop-Omega}(2), it extends continuously as $0$ to $\partial D\setminus\{p\}$. Hence $\Omega^{\alpha_p}_{D,p}$ is continuous in $\overline{D}\setminus\{p\}$.
\end{proof}

\section{Local behavior of the pluricomplex Poisson kernel at the pole}\label{Trap}

\subsection{Lempert's theory} Now we need to briefly recall Lempert's theory \cite{Le1,Le2, Le3}. For all unproven statements, we refer the reader to \cite[Sections 1, 2 and 3]{BPT}.

Let $K\subset \C^n$ be a bounded strongly convex domain with smooth boundary. A {\sl complex geodesic} $\varphi:\D \to  K$ is a holomorphic map which is an isometry between the hyperbolic metric/distance of $\D$ and the Kobayashi metric/distance of $K$. Every complex geodesics in $K$ extends smoothly up to the boundary of $\D$. Given a point $q\in \partial K$ and $v\not\in T_q^\C \partial K$ there exists a complex geodesic $\varphi:\D \to K$ such that $\varphi(1)=q$ and $\varphi'(1)=\lambda v$ for some $\lambda\neq 0$. Moreover, it is unique up to pre-composition with automorphisms of $\D$ fixing $1$. 

To every complex geodesic $\varphi$ is associated a holomorphic retraction $\rho: K \to \varphi(\D)$ such that $\rho\circ \rho=\rho$, and it is the unique holomorphic retraction on $\varphi(\D)$ with affine fibers (\cite[Prop. 3.3]{BPT}), which we call the {\sl Lempert projection}. The Lempert projection $\rho$ is smooth up to $\overline{K}$ and $\rho(z)\in \partial \varphi(\D)$ if and only if $z\in \partial \varphi(\D)$. The map $\tilde\rho:=\varphi^{-1}\circ \rho:K\to \D$ is called the {\sl left-inverse of $\varphi$}.

Lempert \cite[Proposition 11]{Le2} proved that, given any complex geodesic $\varphi$, there exists a biholomorphism $\Phi:K\to \Phi(K)$, which extends as a diffeomorphism up to $\partial K$ such that $\Phi(\varphi(\zeta))=(\zeta,0,\ldots, 0)$, $\Phi(K)$ is strongly convex near $\Phi(\partial \varphi(\D))$, $\Phi(K)\subset \D\times \C^{n-1}$ and the associated Lempert projection $\rho_S$ (that is, $\rho_S=\Phi \circ \rho \circ \Phi^{-1}$) is given by $\rho_S(z)=(z_1,0,\ldots, 0)$. We call $\Phi$ the {\sl Lempert special coordinates adapted to $\varphi(\D)$}. 

\begin{remark}\label{Rem:Lempert-un-grande}
Let $B, W$ be two bounded strongly convex domains with smooth boundary such that $B\subset W$. Assume there exist $p\in \partial B$ and a neighborhood $U$ of $p$ such that $B\cap U=W\cap U$. Let $\varphi:\D \to W$ be a complex geodesic of $W$ such that $\varphi(1)=p$ and suppose that $\varphi(\oD)\subset U$. Let $\rho:W\to \varphi(\D)$ be the Lempert projection. Hence, $\rho|_B:B\to \varphi(\D)$ is a holomorphic retraction and it is then easy to prove that $\varphi$ is a complex geodesic for $B$ as well. It is easy to see that,  if $\Phi$ are Lempert special coordinates for $W$ adapted to $\varphi$, then $\Phi|_B$ are Lempert special coordinates for $B$ adapted to $\varphi$. 
\end{remark} 

Let 
\begin{equation}\label{Eq:Lp}
L_p:=\{v\in \C^n: |v|=1, \langle v, \nu_p\rangle>0, iv\in T_p\partial K\}
\end{equation}
where $\nu_p$ is the outer unit  normal vector to $\partial K$ at $p$ and $\langle \cdot, \cdot \rangle$ is the standard Hermitian product. According to Chang-Hu-Lee \cite{CHL}  for every $v\in L_p$ there exists a unique complex geodesic $\varphi_v:\D \to B$ such that $\varphi_v(1)=p$, $\varphi_v'(1)=\langle v, \nu_p\rangle v$ and $\Im\langle \varphi''(1), \nu_p\rangle =0$. 

\begin{definition}\label{Def:CHL}
A complex geodesic in $K$ parameterized so that $\varphi_v(1)=p$, $\varphi_v'(1)=\langle v, \nu_p\rangle v$ and $\Im\langle \varphi''(1), \nu_p\rangle =0$ is called a {\sl CHL complex geodesic}.
\end{definition}

Given a complex geodesic $\eta:\D \to K$ such that $p\in \eta(\oD)$, there exists an automorphism $h$ of $\D$ such that $\eta\circ h$ is a CHL complex geodesic.

In the sequel we need the following result, which can be seen as a sort of converse of \cite[Prop. 2.5]{BFW}:

\begin{lemma}\label{Lem:tgc}
Let $D\subset \C^n$ be a bounded strongly convex domain with smooth boundary. Let $p\in\partial D$ and let $\nu_p$ be the outer unit normal vector of $\partial D$ at $p$. Let $\{\varphi_j\}$ be a sequence of complex geodesics such that $\varphi_j(1)=p$ for all $j$. Suppose that for every neighborhood $U$ of $p$ there exists $j_U$ such that $\varphi_j(\oD)\subset U$ for all $j>j_U$. Then
\[
\lim_{j\to \infty}\frac{\langle \varphi_j(\zeta)-p, \nu_p\rangle}{|\varphi_j(\zeta)-p|}=0,
\]
uniformly in $\zeta\in \D$.
\end{lemma}
\begin{proof}
Let  $\{ e_j \}_{j=1,\ldots, n}$ denote the canonical base of $\C^n.$ Up to local holomorphic change of coordinate we may  assume that $p = 0$,  $\nu_p = e_{n}$  and that a defining function of $\partial D$ near the origin is of the form $Re(z_n) + h(z_1, \ldots,z_{n-1},\Im z_n)$, where $h$ is a smooth function which is zero up to the first order at $0$.   Let 
\[
\delta_j:= \max\{|\varphi_j(\zeta)|: \zeta\in \oD\}.
\]
Note that, by hypothesis, $\delta_j\to 0$ as $j\to \infty$.
  
In the proof of his Theorem 2 (see  \cite[pag. 409 penultimate displayed formula]{H0}), Huang shows that there exist $j_0$ and a sequence  $\{\gamma_j\}\subset \C^{n-1}\setminus\{0\}$ such that,  writing $\frac{\gamma_j}{|\gamma_j|}:=(a_{1,j}, \ldots,a_{n-1,j})$,  the following equality  holds uniformly with respect to $\zeta \in \D$, for $j\geq j_0$:
\begin{equation}\label{Eq:Huang} 
\frac{ \varphi_j'(\zeta)}{|\gamma_j|} = - (a_{n-1,j}, a_{n-2,j}, \ldots, a_{1,j},0) + O(\delta_j).
\end{equation}
Now,   fix $\xi \in \D$, and $0 \leq t \leq 1$ and let $\zeta = t\xi + (1-t)$. Multiplying by $(\xi-1)$ and integrating \eqref{Eq:Huang} in $t$ between $0$ and $1$,  bearing in mind that $\varphi_j(1) = 0$, we obtain
$$ 
\frac{ \langle \varphi_j(\xi),e_n\rangle }{|\gamma_j|} =  O(\delta_j)(\xi-1)
$$ 
and 
$$ 
\frac{ \varphi_j(\xi)}{|\gamma_j|} =  \left[ - (a_{n-1,j}, a_{n-2,j}, \ldots,a_{1,j},0) + O(\delta_j) \right](\xi-1). 
$$
Hence,
$$ \frac{  |\langle \varphi_j(\xi),e_n\rangle|}{|\varphi_j(\xi)|} = \frac{O(\delta_j)}{1 + O(\delta_j)},
$$
and the statement follows.
\end{proof}

\subsection{Defining couples adapted to a complex geodesic} Now, fix a point $p\in \partial K$. Let  $\varphi_0:\D\to K$ be a complex geodesic such that $\varphi_0(1)=p$, and $v_0:=\varphi_0'(1)$.  Let $\rho:K\to \varphi_0(\D)$ be the associated Lempert projection and $\tilde\rho$ the left-inverse of $\varphi_0$. Let $\Phi$ be the Lempert special coordinates adapted to $\varphi_0$ such that $\varphi_S(\zeta):=\Phi(\varphi_0(\zeta))=(\zeta,0,\ldots, 0)$, $\zeta\in \D$ and $\Phi(p)=e_1:=(1,0,\ldots, 0)$. Note that $d\Phi_p (v_0)=(\Phi\circ \varphi_0)'(1)=e_1$.

Let $K':=\Phi(K)$. Note that $d(\rho_S)_{e_1}(z)=(z_1,0,\ldots, 0)$, hence, the left-inverse $\tilde\rho_S:=\varphi_S^{-1}\circ \rho_S$ satisfies $\tilde\rho_S(z)=z_1$. Therefore,
\[
T^\C_{e_1}\partial K'=\ker d(\tilde\rho_S)_{e_1}, \quad T_{e_1}\partial K'=\ker \Re d(\tilde\rho_S)_{e_1}.
\]
Since $d\Phi_p$ maps $T^\C_{p}\partial K$ onto $T^\C_{e_1}\partial K'$ and $T_{p}\partial K'$ onto $T_{e_1}\partial K'$, taking into account that $\tilde \rho=\tilde\rho_S\circ \Phi$, we see that
\begin{equation}\label{kernel-tg}
T^\C_{p}\partial K=\ker d\tilde\rho_{p}, \quad T_{p}\partial K=\ker \Re d\tilde\rho_{p}.
\end{equation}
\[
\theta: \C^n\to \R, \quad \theta(v):=d\tilde\rho_p(v).
\]
Let $\alpha:=\Im \theta|_{T_p\partial K}$. Then,  by \eqref{kernel-tg}, $\ker \alpha=T_p^\C\partial K$. Therefore, $(\alpha,\theta)$ is a defining couple of $T_p^\C\partial K$. 
\begin{definition}
The couple $(\alpha,\theta)$ defined above is the {\sl defining couple adapted to $\varphi_0$}.
\end{definition}

Thus we can define the pluricomplex Poisson kernel $\Omega^\alpha_{K,p}$. Now, let $\varphi:\D\to K$  be another complex geodesic such that $\varphi(1)=p$. By Hopf's Lemma,  $\varphi'(1)\not\in T_p\partial K$. Hence, 
\begin{equation}\label{Eq:curve-special-CHL}
\lim_{r\to 1^-}\Omega^\alpha_{K,p}(\varphi(r))(1-r)=-2\Re \frac{1}{\theta(\varphi'(1))}=-2\Re \frac{1}{d\tilde\rho_p(\varphi'(1))}.
\end{equation}
By Proposition \ref{ovvia}, there exists a constant $m>0$ such that $\Omega^\alpha_{K,p}=m\Omega_{K,p}$, where $\Omega_{K,p}$ is the pluricomplex Poisson kernel constructed in \cite{B-P, BPT}---and it is, in fact, the pluricomplex Poisson kernel associated to the linear form $T_p\partial K\ni v\mapsto \langle v, \nu_p\rangle$, where $\nu_p$ is the outer unit  normal vector to $\partial K$ at $p$. 

By construction (see \cite[(1.2)]{BPT}) $\Omega_{K,p}$ restricted to the complex geodesics whose closure contain $p$ is a multiple of the Poisson kernel $P_\D$. Thus, by the Phragm\'en-Lindel\"of theorem in~$\D$,
\begin{equation}\label{Eq:a-v}
\Omega^\alpha_{K,p}(\varphi(\zeta))=\Re\left(\frac{1}{d\tilde\rho_p(\varphi'(1))}\right) P_\D(\zeta).
\end{equation}

\subsection{Entrapping strongly pseudoconvex domains between strongly convex domains}\label{subset-tecn} Assume that $D\subset \C^n$ is a bounded, strongly pseudoconvex domain with smooth boundary. Let $p\in \partial D$. By \cite[Thm. 2.6]{BFW} there exist a biholomorphism $\Phi:D\to \Phi(D)$, which extends as a smooth diffeomorphism on the boundary, a bounded strongly convex domain $W\subset \C^n$ with smooth boundary, and an open neighborhood $U$ of $\Phi(p)$ such that $\Phi(D)\subset W$ and $U\cap W=U\cap \Phi(D)$. Moreover, it is clear that we find a smooth bounded strongly convex domain $B\subset \C^n$ such that $B\subset \Phi(D)$ and $B\cap U=\Phi(D)\cap U$. Therefore, up to considering $\Phi(D)$ instead of $D$, we  assume that

\bigskip

\begin{itemize}
\item[(H1)] There exist two bounded strongly convex domains $B,W\subset \C^n$ with smooth boundary and an open neighborhood $U$ of $p$ such that $B\subset D\subset W$ and $U\cap D=U\cap W=U\cap B$.
\end{itemize}

\bigskip

Note that 
\[
A:=T_p\partial D=T_p\partial W=T_p\partial B
\]
 and that 
 \[
A^\C:= T^\C_p\partial D=T^\C_p\partial W=T^\C_p\partial B.
 \]

 The following lemma follows the ideas in \cite[Proof of Thm. 2.6]{BFW} and it is based on a ``preservation of geodesics'' result by X. Huang \cite[Corollary 1]{H}. For a vector $v\in \C^n$ we let $v^T$ the orthogonal projection onto $A^\C$ and  $v^N=v-v^T$.

\begin{lemma}\label{Lemma:geo}
There exists $\epsilon>0$ such that if $v\in A\setminus A^\C$,  is such that $|v^N|<\epsilon |v^T|$ then any complex geodesic $\varphi:\D \to W$ such that $\varphi(1)=p$ and $\varphi'(1)=\lambda v$ for some $\lambda\neq 0$ is a complex geodesic---that is, an isometry for both the Kobayashi distance and Kobayashi metric---also for $D$ and for $B$. Moreover, if $\rho:W\to \varphi(\D)$ is the Lempert projection associated to $\varphi$ (as a complex geodesic in $W$), then $\rho|_B$ is the Lempert projection associated to $\varphi$ (as a complex geodesic in $B$). Also,  $\rho|_D$ is a holomorphic retraction of $D$ onto $\varphi(\D)$ with affine fibers.
 \end{lemma}
 \begin{proof}
Let $U'$ be an open neighborhood  of $p$ such that $\overline{U'}\subset U$. By \cite[Corollary 1]{H}, there exists $\epsilon>0$ such that if $v\in A\setminus A^\C$,  is such that $|v^N|<\epsilon |v^T|$ then any complex geodesic $\varphi:\D \to W$ such that $\varphi(1)=p$ and $\varphi'(1)=\lambda v$ for some $\lambda\neq 0$ is such that $\varphi(\oD)\subset U'$. Let $\rho:W\to \varphi(\D)$ be the associated Lempert projection. Then $\rho(W)=\varphi(\D)\subset U'\subset B\subset D$. Hence, $\varphi(\D)$ is a holomorphic retract of $B$ and of $D$. Since it is  well known that that one-dimensional holomorphic retracts of strongly convex domains are complex geodesics,  $\varphi$ is a complex geodesic in $B$.  By the same token, or by the monotonicity property of the Kobayashi distance, it follows that $\varphi$ is an isometry for both the Kobayashi metric and the Kobayashi distance of $D$ as well. Finally, since the fibers of $\rho|_B$ are affine, by \cite[Proposition~3.3]{BPT}, it is the Lempert projection associated to $\varphi$ as a complex geodesic in $B$.
 \end{proof}

\begin{definition}
Let $v\in A\setminus A^\C$ be a vector. We say that $v_0$ is {\sl compatible} if $|v^N|<\epsilon |v^T|$, where $\epsilon>0$ is given by Lemma~\ref{Lemma:geo}.

A holomorphic map $\varphi:\D \to \C^n$, which extends smoothly on $\partial \D$ and is a complex geodesic (that is, an isometry for both the Kobayashi distance and the Kobayashi metric) in either $B$, $D$ or $W$ and such that $\varphi(1)=p$ and $\varphi'(1)$ is compatible at $p$, is a {\sl compatible} complex geodesic. 

The defining couple of $A^\C$ associated to $\varphi$ is called a {\sl compatible defining couple}.
\end{definition}

The definition of compatible complex geodesics needs a couple of comments. Holomorphic discs which are isometries for the Kobayashi metric are usually called ``infinitesimal complex geodesics''. In bounded convex domains, complex geodesics and infinitesimal complex geodesics are one and the same. In other domains, it is not known. However, for the aim of this paper, it is convenient to call ``complex geodesic'' any holomorphic disc which is both an isometry for the Kobayashi metric and for the Kobayashi distance. 

Also, every complex geodesic in $B$ and $W$ extends smoothly on the boundary. However, this is not granted for complex geodesics in $D$, but it is part of the definition of ``compatibility''. 

The following lemma explains why we choose to use a unified terminology for $B, W$ and $D$:

\begin{lemma}
A holomorphic map $\varphi:\D \to \C^n$ is a compatible complex geodesic in $B$ if and only if it is a a compatible complex geodesic in $D$ if and only if it is a a compatible complex geodesic in $W$. In particular, every compatible complex geodesic in $D$ is the image of a  holomorphic retract of $D$ with affine fibers.
\end{lemma}
\begin{proof}
By Lemma~\ref{Lemma:geo}, if $\varphi$ is a compatible complex geodesic in $W$, then it is a compatible complex geodesic in $B$ and $D$ (and it is the image of a holomorphic retract of $D$ with affine fibers). 

If $\varphi$ is compatible in $B$, let $\eta:\D \to W$ be the complex geodesic in $W$ such that $\eta(1)=1$ and $\eta'(1)=\lambda \varphi'(1)$ for some $\lambda\neq 0$. Then, $\eta$ is compatible in $W$, hence, by Lemma~\ref{Lemma:geo}, it is also a complex geodesic in $B$ (and in $D$) and, by the uniqueness up to re-parametrization of complex geodesics in strongly convex domains, it follows that $\varphi=\eta\circ h$ for a suitable automorphism $h$ of $\D$ such that $h(1)=1$. Hence, $\varphi$ is a compatible complex geodesic in $W$ and $D$ as well.

Finally, if $\varphi$ is a compatible complex geodesic in $D$, by definition it is also an infinitesimal complex geodesic and, being smooth up to the boundary,  the hypothesis in Huang's preservation principle \cite[Corollary~1]{H} is satisfied. Therefore, $\varphi(\D)\subset U'$ (where $U'$ is defined in the proof of Lemma~\ref{Lemma:geo}). In particular, $\varphi(\D)\subset B$. By the monotonicity of the Kobayashi distance, it follows that $\varphi$ is a (compatible) complex geodesic of $B$ as well, and, for what we already proved, for $W$.
\end{proof}

In the sequel we let $\varphi_0:\D \to B$ be a complex geodesic such that $\varphi_0(1)=p$ and $v_0:=\varphi_0'(1)$ is a compatible vector, that is, $\varphi_0$ is a compatible complex geodesic. Let $\rho$ be Lempert projection associated with $\varphi_0$ in $W$. By Lemma~\ref{Lemma:geo}, $\rho|_B$ is the Lempert projection of $\varphi_0$ as complex geodesic in $B$, hence, $\tilde\rho:=\varphi^{-1}\circ \rho$ is the left-inverse of $\varphi_0$ (in both $B$ and $W$). 

Let $\theta:=d\tilde\rho_p$ and let $\alpha:=\Im \theta|_A$. Then $(\alpha,\theta)$ is a compatible defining couple of $T^\C_p\partial D$.

For all $z\in B$,
\begin{equation}\label{Eq:dineq}
\Omega^\alpha_{W,p}(z)\leq \Omega^\alpha_{D,p}(z)\leq \Omega^\alpha_{B,p}(z).
\end{equation}

As a corollary of the previous inequalities we generalize \cite[Corollary 5.3]{BPT}:
\begin{proposition}\label{tg-lim}
Let $\{z_n\}\subset D$ be a sequence converging to $p$ and assume that $\lim_{n\to \infty}\frac{p-z_n}{|p-z_n|}=v\not\in A^\C$. Then
\begin{equation}\label{Eq:stef-serve}
\lim_{n\to \infty}\Omega^\alpha_{M,p}(z_n)|\langle p-z_n, \nu_p\rangle |=-2\Re \left(\frac{1}{\theta (\nu_p)}\right)\frac{\Re \langle v, \nu_p\rangle}{|\langle v, \nu_p\rangle|},
\end{equation}
where $M=B, D, W$ and $\nu_p$ is the outer unit  normal vector of $M$ at $p$.
\end{proposition}
\begin{proof}
In view of \eqref{Eq:dineq}, it is enough to prove the formula for $B$ and $W$. We prove it for $B$, being the proof exactly the same for $W$.

 For each $n$ let $\varphi_n$ be a CHL complex geodesic (see Definition~\ref{Def:CHL}) such that $\varphi_n(1)=p$ and $\varphi_n(\zeta_n)=z_n$ for some $\zeta_n\in \D$. Up to subsequences, we can assume that $\beta=\lim_{n\to \infty}\frac{1-\zeta_n}{|1-\zeta_n|}$ exists. 

Up to extracting further subsequences, we may also assume that $\{\varphi_n\}$ converges uniformly on compacta to a holomorphic map $\varphi:\D \to \overline{B}$. Moreover, since $B$ is strongly (pseudo)convex, then, either $\varphi(\D)\subset B$ or there exists a point $q\in \partial B$ such that $\varphi(\D)=q$. In the latter case, since $z_n\in \varphi_n(\D)$ and $z_n\to p$, by \cite[Prop. 2.3]{BFW}, it follows that $q=p$ and that for every neighborhood $Q$ of $p$ there exists $n_Q$ such that $\varphi_n(\oD)\subset Q$ for all $n>n_Q$. Hence, by Lemma~\ref{Lem:tgc}, 
\[
\lim_{n\to \infty}\frac{\langle z_n-p, \nu_p\rangle}{|z_n-p|}=\lim_{n\to \infty}\frac{\langle \varphi_n(\zeta_n)-p, \nu_p\rangle}{|\varphi_n(\zeta_n)-p|}=0,
\] 
which implies $v\in A^\C$, a contradiction. Therefore, $\varphi(\D)\subset B$.

Now,  since the $\varphi_n$'s  are isometries between the hyperbolic distance in $\D$ and the Kobayashi distance in $B$, then $\varphi$ is a complex geodesic in $B$ and by \cite[Prop. 1]{H} (see also \cite[Section 2]{BPT}), $\{\varphi_n\}$  converges   in the $C^1$-topology  on $\overline{\D}$ to $\varphi$.

Write $(\varphi_n)'(1)=\langle v_n, {\nu_p}\rangle v_n$ for some $v_n\in L_p$. In particular notice that, since $\langle v_n, {\nu_p}\rangle>0$ by definition, we have $\Re \langle (\varphi_n)'(1),\nu_p\rangle=\langle (\varphi_n)'(1),\nu_p\rangle>0$ for all $n$. Since by Hopf's Lemma, $\Re \langle \varphi'(1), {\nu_p}\rangle\neq 0$, it follows easily from the previous considerations that $\varphi$ is CHL and hence $\varphi'(1)=\langle \tilde v, \nu_p\rangle \tilde v$ for some $\tilde v\in L_p$ so that $\lim_{n\to \infty}v_n=\tilde v$. 
By \eqref{Eq:a-v}, recalling that $\theta=d\tilde\rho_p$ where $\tilde\rho$ is the left-inverse of the chosen compatible geodesic $\varphi_0$,
\begin{equation*}
\begin{split}
\Omega^\alpha_{B,p}(z_n)|\langle p-z_n, \nu_p\rangle |& =\Omega^\alpha_{B,p}(\varphi_n(\zeta_n))|\langle p-z_n, \nu_p\rangle |\\&=-\Re\left(\frac{1}{\theta(\varphi_n'(1))}  \right) \Re\left(\frac{(1+\zeta_n)|1-\zeta_n|}{1-\zeta_n}\right) \left|\langle \frac{p-z_n}{1-\zeta_n}, \nu_p\rangle \right|.
\end{split}
\end{equation*}
By the Fundamental Theorem of Calculus,
\begin{equation*}
\frac{p-z_n}{1-\zeta_n}=\frac{\varphi_n(1)-\varphi_n(\zeta_n)}{1-\zeta_n}=\int_0^1 (\varphi_n)'((1-\zeta_n)t+\zeta_n) dt,
\end{equation*}
and, since the integrand converges uniformly to $\varphi'(1)$, and taking into account that $\theta(\varphi'(1))= (\langle \tilde v, \nu_p\rangle)^2 \theta(\nu_p)$---since $\theta|_{A^\C}=0$,
the previous equation implies that
\begin{equation}\label{Poi-beta-g}
\lim_{n\to\infty}\Omega^\alpha_{B,p}(z_n)|\langle p-z_n, \nu_p\rangle | =-2 \Re\left(\frac{1}{\theta(\nu_p)}\right)\Re \beta.
\end{equation}
On the other hand, 
\begin{equation*}
\langle v,\nu_p\rangle=\lim_{n\to \infty}\left[\frac{1-\zeta_n}{|1-\zeta_n|}\frac{|1-\zeta_n|}{|\varphi_n(1)-\varphi_n(\zeta_n)|}\, \langle \frac{\varphi_n(1)-\varphi_n(\zeta_n)}{1-\zeta_n}, \nu_p \rangle\right]=\beta \langle \tilde v, \nu_p\rangle.
\end{equation*}
Putting together \eqref{Poi-beta-g} and the previous formula, one easily sees that \eqref{Eq:stef-serve} holds.
\end{proof}

\begin{lemma}\label{Lemma:M1}
Let $\{z_n\}\subset B$ be a sequence converging to $p$ such that $v:=\lim_{n\to \infty}\frac{p-z_n}{|p-z_n|}$ exists. If either $v\in A^\C$ or $|v^N|<\frac{\epsilon}{2} |v^T|$, where $\epsilon>0$ is given by Lemma~\ref{Lemma:geo}, then there exists $n_0$ such that for all $n>n_0$
\[
\frac{\Omega^\alpha_{W,p}(z_n)}{\Omega^\alpha_{B,p}(z_n)}=1.
\]
In particular, for all $n>n_0$,
\[
\Omega^\alpha_{D,p}(z_n)=\Omega^\alpha_{B,p}(z_n)=\Omega^\alpha_{W,p}(z_n).
\]
\end{lemma}
\begin{proof}
It is clear that the second equation follows at once from the first one and from \eqref{Eq:dineq}. In order to prove the first equation, let $\{z_n\}\subset B$ be a sequence converging to $p$.  By hypothesis, the complex geodesic $\varphi_n:\D \to B$ such that $\varphi_n(1)=p$ and $\varphi_n(0)=z_n$ eventually satisfies $|[\varphi_n'(0)]^N|\leq \epsilon |[\varphi_n'(0)]^T|$. Therefore, $\varphi_n$ is a complex geodesic for both $B$ and $W$ eventually by Lemma~\ref{Lemma:geo}.  Thus, by \eqref{Eq:a-v},
$\Omega_{B,p}^\alpha(\varphi_n(\zeta))=\Omega_{W,p}^\alpha(\varphi_n(\zeta))$ for $n$ large, and for all $\zeta\in \D$, and we are done.
\end{proof}

The argument used in the previous proof allows us to prove:
\begin{proposition}\label{Prop:lle}
For all compatible complex geodesics $\varphi$ and $\zeta\in \D$,
\begin{equation}\label{Eq:restr-PD}
\Omega^\alpha_{D,p}(\varphi(\zeta))=\Re\left(\frac{1}{\theta(\varphi'(1))}\right) P_\D(\zeta).
\end{equation}
In particular, there exists an open set $J$ whose closure contains $p$ such that $\Omega^\alpha_{D,p}$ is $C^\infty$-smooth on $J$, $(dd^c\Omega^\alpha_{D,p})^{n-1}\neq 0$ on $J$ and the associated Monge-Amp\'ere foliation on $J$ is formed by complex geodesics for $D$, which extend smoothly up to the boundary of $\D$, and have an associated  holomorphic retraction with affine fibers.
 \end{proposition}
 \begin{proof}
 The first statement follows at once from the previous proof and equations \eqref{Eq:dineq} and \eqref{Eq:a-v}. 

The set $G:=\{v\in L_p: |v^N|<\frac{\epsilon}{2}|v^T|\}$,  is open in $L_p$ (see \eqref{Eq:Lp}).  The map $G\times \D \ni (v,\zeta)\mapsto \varphi_v(\zeta)$ is a homeomorphism onto its image, which is thus an open set $J$ (see \cite[Corollary 2.3]{BPT}). Hence,  by \eqref{Eq:a-v} and \eqref{Eq:restr-PD}, we have $\Omega_{D,p}^\alpha(z)=\Omega_{B,p}^\alpha(z)$ for all $z\in J$ and we are done. \end{proof}

 \begin{proposition}\label{Prop:M2}
The following hold:
\[
\lim_{z\to p}\frac{\Omega^\alpha_{W,p}(z)}{\Omega^\alpha_{B,p}(z)}=1.
\]
In particular,
\[
\lim_{z\to p}\frac{\Omega^\alpha_{D,p}(z)}{\Omega^\alpha_{B,p}(z)}=\limsup_{z\to p}\frac{\Omega^\alpha_{D,p}(z)}{\Omega^\alpha_{W,p}(z)}=1.
\]
\end{proposition}
\begin{proof}
Again, the second formula follows immediately from the first one and from  \eqref{Eq:dineq}. By Lemma~\ref{Lemma:M1}, we are left to consider sequences $\{z_n\}\subset B$ such that $\frac{p-z_n}{|p-z_n|}\to v$ with either $v\not\in A$ or $v\in A\setminus A^\C$. 

\medskip

{\sl Case $v\not\in A$}. By \eqref{tg-lim}, and taking into account that $\Re \langle v, \nu_p\rangle >0$ (since $v\not\in A$), we have
\[
\lim_{n\to \infty}\frac{\Omega^\alpha_{B,p}(z_n)}{\Omega^\alpha_{W,p}(z_n)}=\lim_{n\to \infty}\frac{\Omega^\alpha_{B,p}(z_n)}{\Omega^\alpha_{W,p}(z_n)}\frac{|\langle p-z_n, \nu_p\rangle |}{|\langle p-z_n, \nu_p\rangle |}=1,
\]
and we are done.

\medskip
{\sl Case $v\in A\setminus A^\C$}.   Let $\{z_n\}\subset B$ be such that $\frac{p-z_n}{|p-z_n|}\to v$ with $v\in A\setminus A^\C$. 

Let $\varphi_n^B:\D \to B$ be the CHL complex geodesic in $B$ such that $\varphi_n^B(1)=p$  and $\varphi_n^B(\zeta_n^B)=z_n$ for some $\zeta_n^B\in \D$ and similarly let $\varphi_n^W:\D \to W$ be the CHL complex geodesic in $W$ such that $\varphi_n^W(1)=p$ and $\varphi_n^W(\zeta_n^W)=z_n$ for some $\zeta_n^W\in \D$. 

Up to extracting subsequences, we may assume that $\{\varphi_n^B\}$ converges uniformly on compacta to a holomorphic map $\varphi^B:\D \to \overline{B}$ and $\{\varphi_n^W\}$ converges uniformly on compacta to a holomorphic map $\varphi^W:\D \to \overline{W}$. Arguing as in the proof of Proposition~\ref{tg-lim}, one can see that $\varphi^B$ and  $\varphi^W$ are CHL complex geodesics  and  $\{(\varphi^B_n)\}$  converges to $\varphi^B$ while $\{(\varphi^W_n)\}$ converges  to $\varphi^W$ in the $C^1$-topology on $\overline{\D}$. 

In particular, we can write $(\varphi^B)'(1)=\langle v^B, \nu_p\rangle v^B$ and $(\varphi^W)'(1)=\langle v^W, \nu_p\rangle v^W$ for some $v^B, v^W\in L_p$. 

Using the Chang-Hu-Lee spherical representation \cite{CHL} (see also \cite[Section 1]{BPT}), it is easy to see that $\lim_{n\to \infty}\zeta_n^B=\lim_{n\to \infty}\zeta_n^W=1$. By the Fundamental Theorem of Calculus,
\[
\frac{p-z_n}{|p-z_n|}=\frac{1-\zeta_n^B}{|p-z_n|}\int_0^1 (\varphi^B_n)'((1-\zeta_n^B)t+\zeta_n^B) dt.
\]
Since the integrand converges uniformly in $n$ to $(\varphi^B)'(1)=\langle v^B,\nu_p\rangle v^B$ and the left hand side converges to $v$, it follows that $\frac{1-\zeta_n^B}{|p-z_n|}$ converges to some $\lambda^B\neq 0$, as $n\to \infty$. A similar argument for $\varphi_n^W$ shows that
\[
\lambda^B \langle v^B,\nu_p\rangle v^B=v =\lambda^W \langle v^W,\nu_p\rangle v^W.
\]
Bearing in mind that $v^B, v^W\in L_p$, this implies immediately that $v^B=v^W$. 

In particular, $\varphi^B(1)=\varphi^W(1)=p$ and $(\varphi^B)'(1)=(\varphi^W)'(1)$. By \eqref{Eq:a-v} and \eqref{Eq:dineq}, we have
\begin{equation}\label{Eq:P-po}
1\leq \frac{\Omega_{W,p}^\alpha(z_n)}{\Omega_{B,p}^\alpha(z_n)}=\frac{P_\D(\zeta_n^W)}{P_\D(\zeta_n^B)}.
\end{equation}
Let $\tilde\rho^B_n$ be the left-inverse of $\varphi_n^B$. By \cite[Lemma 3.5]{BPT}, $\{\tilde\rho_n^B\}$ converges in the $C^1$-topology on $\overline{B}$ to the left-inverse $\tilde\rho^B$ of $\varphi^B$. 

Since $B$ and $W$ have common boundary near $p$, there exists $r\in (0,1)$ so that $Q:=\D\cap \{\zeta\in \C: |\zeta-1|<r\}\subset (\varphi_n^W)^{-1}(B\cap \varphi_n^W(\D))$ for all $n$. We can assume that $\zeta_n^W\in Q$ for all $n$. Let $g:\D\to Q$ be a Riemann map. Note that since $\partial Q=\partial \D$ near $1$, we can assume that $g$ extends holomorphically through $\partial \D$ near $1$, $g(1)=1$ and (up to pre-composing $g$ with a hyperbolic automorphism of $\D$ fixing $1$) that $g'(1)=1$. Let $\tilde\zeta_n^W:=g^{-1}(\zeta_n^W)$. Note that $\tilde\zeta_n^W\to 1$ as $n\to \infty$. Let $f_n:=\tilde\rho_n^B\circ \varphi_n^W\circ g: \D \to \D$. 

By construction, $f_n$ is smooth up to $\partial \D$ close to $1$, $f_n(1)=1$ and $f_n(\tilde\zeta_n^W)=\zeta_n^B$. By the classical Julia's Lemma  and the classical Julia-Wolff-Carath\'eodory Theorem  (see, {\sl e.g.}, \cite[Thm. 1.2.5 and Thm. 1.2.7]{Ababook} or \cite[Thm. 1.7.3 and Lemma 1.4.5]{BCDbook}), $f_n'(1)>0$ and for all $\zeta\in \D$,
\begin{equation}\label{Julia-classn}
\frac{|1-f_n(\zeta)|^2}{1-|f_n(\zeta)|^2}\leq f_n'(1) \frac{|1-\zeta|^2}{1-|\zeta|^2}.
\end{equation}

Moreover, $\{f_n\}$ converges in the $C^1$-topology on $\oD$ to $f:=\tilde\rho^B\circ \varphi^W\circ g$. Note that $f'(1)=d\tilde\rho^B_p((\varphi^W)'(1))g'(1)=d\tilde\rho^B_p((\varphi^B)'(1))=1$. 

Hence, by \eqref{Julia-classn},
\begin{equation*}
\begin{split}
\frac{P_\D(\zeta_n^W)}{P_\D(\zeta_n^B)}&=\frac{P_\D(g(\tilde\zeta_n^W))}{P_\D(f_n(\tilde\zeta_n^W))}=\frac{1-|g(\tilde\zeta_n^W)|^2}{|1-g(\tilde\zeta_n^W)|^2}\frac{|1-\tilde\zeta_n^W|^2}{1-|\tilde\zeta_n^W|^2}\frac{1-|\tilde\zeta_n^W|^2}{|1-\tilde\zeta_n^W|^2}\frac{|1-f_n(\tilde\zeta_n^W)|^2}{1-|f_n(\tilde\zeta_n^W)|^2}\\ &\leq f'_n(1)\frac{1-|g(\tilde\zeta_n^W)|^2}{|1-g(\tilde\zeta_n^W)|^2}\frac{|1-\tilde\zeta_n^W|^2}{1-|\tilde\zeta_n^W|^2}.
\end{split}
\end{equation*}
Therefore, since $f_n'(1)\to f'(1)=1$ as $n\to \infty$, 
\[
\limsup_{n\to \infty}\frac{P_\D(\zeta_n^W)}{P_\D(\zeta_n^B)}\leq \limsup_{n\to \infty}\frac{1-|g(\tilde\zeta_n^W)|^2}{|1-g(\tilde\zeta_n^W)|^2}\frac{|1-\tilde\zeta_n^W|^2}{1-|\tilde\zeta_n^W|^2}. 
\]

Since $g$ extends holomorphically near $1$  and $g'(1)=1$, we see that 
\[
\lim_{n\to \infty}\frac{|1-\tilde\zeta_n^W|^2}{|1-g(\tilde\zeta_n^W)|^2}=1.
\]
On the other hand, in order to compute $\lim_{n\to \infty}\frac{1-|g(\tilde\zeta_n^W)|^2}{1-|\tilde\zeta_n^W|^2}$, we use the Cayley transform $C:\{w\in \C : \Re w>0\}\to \D$ defined as $C(w)=\frac{1-w}{1+w}$. Let $w_n:=C^{-1}(\zeta_n^W)$. Then $\{w_n\}$ converges to $0$. Let $\tilde g:=C^{-1}\circ g \circ C$. The map $\tilde g$ is well defined and holomorphic near $0$, $\tilde g(0)=0$,  $\Re \tilde g(iy)=0$ for $y\in \R$ close to $0$ and $\tilde g'(0)=1$. Hence, $\Re \tilde g(w)=(\Re w)a(w)$ for $w$ close to $0$, where $a(w)$ is a real analytic function such that $a(0)=1$.  

Since $1-|C(w)|^2=\frac{4}{|1+w|^2}\Re w$, it follows that
\[
\frac{1-|g(\tilde\zeta_n^W)|^2}{1-|\tilde\zeta_n^W|^2}=\frac{|1+w_n|^2}{|1+\tilde g(w_n)|^2}\frac{\Re \tilde g(w_n)}{\Re w_n}=\frac{|1+w_n|^2}{|1+\tilde g(w_n)|^2}a(w_n),
\]
hence,  $\lim_{n\to \infty}\frac{1-|g(\tilde\zeta_n^W)|^2}{1-|\tilde\zeta_n^W|^2}=1$.

Therefore, $\limsup_{n\to\infty}\frac{P_\D(\zeta_n^W)}{P_\D(\zeta_n^B)}\leq 1$ and, by \eqref{Eq:P-po}, we are done.
\end{proof}

\section{Julia's lemma for strongly pseudoconvex domains in Stein manifolds}

\begin{definition}
Let $M$ be a Stein manifold and let $D\subset\subset M$ be a
strongly pseudoconvex domain with smooth boundary. Let $p\in \de D$ and let $(\alpha, \theta)$ be a defining couple for $T^\C_p\partial D$. The {\sl horoball} of center $p$ and radius  $R>0$ is
\[
H_\alpha^D(p,R):=\{z\in D: \Omega^\alpha_{D,p}(z)<-1/R\}.
\]
\end{definition}

By Proposition~\ref{ovvia}, horoballs are independent of the defining couple $(\alpha, \theta)$, in the sense that, if $(\alpha', \theta')$ is another defining couple, there exists $\lambda>0$ such that $H_\alpha^D(p,R)=H_{\alpha'}^D(p,\lambda R)$ for all $R>0$.

Note that $H^D_\alpha(p,R)$ is an open subset of $D$ and that $\overline{H^D_\alpha(p,R)}\cap \de D=\{p\}$ for all $R>0$.

\begin{definition}
Let $D\subset\subset M$ be a
strongly pseudoconvex domain with smooth boundary.   A sequence $\{z_n\}\subset D$ is an {\sl $E_p$-sequence} if $\{z_n\}$ converges to $p\in \partial D$ and it is eventually contained in any horoball at $p$ for some---and hence any---defining couple for $T_p^\C\partial D$. 
\end{definition}

Note that every sequence $\{z_n\}\subset D$ which converges non-tangentially to $p\in \partial D$ is an $E_p$-sequence. Indeed, choosing local coordinates around $p$, one can find a ball $\B_r\subset D$ which is tangent to $\partial D$ at $p$. If $\{z_n\}$ converges non tangentially to $p$ then $\{z_n\}$ is eventually contained in $\B_r$ and converges to $p$ non-tangentially in $\B_r$ as well. Hence, $\Omega^\alpha_{D,p}(z_n)\leq \Omega^\alpha_{\B_r,p}(z_n)\to -\infty$ as $n\to \infty$.

\begin{theorem}\label{Julia}
Let $M, M'$ be  Stein manifolds. Let $D\subset\subset M$ and  $D'\subset\subset M'$ be a
strongly pseudoconvex domains with smooth boundary.  Let $f: D \to D'$ be holomorphic. For every $p\in \partial D$ fix a defining couple $(\alpha_p, \theta_p)$ of $T_p^\C\partial D$, and similarly for every $q\in \partial D'$ fix a defining couple $(\alpha'_p, \theta'_p)$ of $T_p^\C\partial D'$. For $p\in \partial D$ and $q\in \partial D'$, let \begin{equation*}
\lambda_{p,q}:=\sup_{z\in D}\left\{\frac{\Omega^{\alpha_p}_{D,p}(z)}{\Omega^{\alpha'_q}_{D',q}(f(z))}\right\}.
\end{equation*}
Suppose there exists $p\in\partial D$ such that
\[
\lambda_p:=\inf_{q\in \de D'}\lambda_{p,q}<+\infty.
\]
Then there exists a unique $q\in \de D'$ such that $\lambda_{p,q}<+\infty$ and $\lambda_p=\lambda_{p,q}$. Moreover, for every $E_p$-sequence $\{z_n\}$, the sequence $\{f(z_n)\}$ is an $E_q$-sequence---in particular, $f$ has non-tangential-limit $q$ at $p$---and for all $R>0$
\[
f(H_{\alpha_p}^D(p,R))\subseteq H_{\alpha'_q}^{D'}(q, \lambda_p R).
\]
\end{theorem}
\begin{proof}
In order to avoid burdening notations, we omit to write $\alpha_p, \alpha_q'$. By hypothesis, there exists $q\in \de D'$ such that 
\[
\lambda_{p,q}:=\sup_{z\in D}\left\{\frac{\Omega_{D,p}(z)}{\Omega_{D',q}(f(z))}\right\}<+\infty.
\]
Let $\{z_n\}\subset D$ be an $E_p$-sequence.  Hence $\lim_{n\to \infty}\Omega_{D,p}(z_n)=-\infty$ and since $\lambda_{p,q}<+\infty$, it follows immediately that $\lim_{n\to \infty}\Omega_{D',q}(f(z_n))=-\infty$. Thus 
 $\lim_{n\to \infty}f(z_n)=q$ and $\{f(z_n)\}$ is eventually contained in any horoball of $D'$ at $q$.

Finally, if it were $\lambda_{p,q'}<+\infty$ for some other $q'\in\partial D'\setminus\{q\}$, arguing as before it would follows that $\{f(z_n)\}$ is both an $E_{q}$-sequence and an $E_{q'}$-sequence for every $E_p$-sequence $\{z_n\}$, a contradiction. 
\end{proof}

It is interesting to note that the condition $\lambda_p<+\infty$ in the previous theorem is, in fact, a local condition at $p$ along non-tangential directions:

\begin{proposition}\label{local-Julia}
Let $M, M'$ be  Stein manifolds. Let $D\subset\subset M$ ({\sl respectively} $D'\subset\subset M'$) be a
strongly pseudoconvex domain with smooth boundary.  Let $f: D \to D'$ be holomorphic. Then $\lambda_p<+\infty$ if and only if there exist $C>0$ and $q\in \partial D'$ such that for all sequences $\{z_n\}$ converging non-tangentially to $p$,
\begin{equation}\label{local-J}
\limsup_{n\to \infty}\frac{\Omega^{\alpha_p}_{D,p}(z_n)}{\Omega^{\alpha'_q}_{D',q}(f(z_n))}\leq C.
\end{equation}
Moreover, $\lambda_p\leq C$.
\end{proposition}
\begin{proof}
Suppose \eqref{local-J} holds. Let $u(z):= C\Omega^{\alpha'_q}_{D',q}(f(z))$. Note that $u$ is plurisubharmonic  and $u<0$ in $D$. Moreover, 
let $\gamma:[0,1]\to D\cup\{p\}$ be a smooth curve such that $\gamma(t)\in D$ for $t\in [0,1)$, $\gamma(1)=p$ and
$\gamma'(1)\not\in T_p\de D$. Then 
\[
\limsup_{t\to 1^-}u(\gamma(t))(1-t)\leq \limsup_{t\to 1^-}\Omega^{\alpha_p}_{D,p}(\gamma(t))(1-t)=-2\Re [\theta_p(\gamma'(1)]^{-1}.
\]
Therefore, $u(z)\in \mathcal S_{\alpha_p}(D)$ and hence, $u(z)\leq  \Omega^{\alpha_p}_{D,p}(z)$ for all $z\in D$.
\end{proof} 

\section{An intrinsic  Julia-Wolff-Carath\'eodory Theorem for strongly pseudoconvex domains in $\C^n$}

Let $D\subset \C^n$ be a bounded, strongly pseudoconvex domain with smooth boundary. 

For each $p\in \partial D$, we can globally change coordinates so that  assumption  (H1) in Subsection~\ref{subset-tecn} holds. So, for each $p\in \partial D$ we can select a compatible complex geodesic $\varphi^D_p:\D \to D$ such that $\varphi^D_p(1)=p$ and $v^D_p:=(\varphi^D_p)'(1)$. We let $\rho_p^D$ be the associated Lempert's projection and $\tilde\rho^D_p$  the left-inverse of $\varphi^D_p$. We have thus the (compatible) defining couple $(\Im (d\tilde\rho^D_p)_p|_{T_p\partial D}, (d\tilde\rho^D_p)_p)$. 

\begin{theorem}\label{Thm:JWC}
Let  $D\subset \C^n, D'\subset \C^m$ be bounded, strongly pseudoconvex domains with smooth boundary. For each $p\in \partial D$ ({\sl respectively}, $q\in \partial D'$) let $(\alpha_p, \theta_p)$ ({\sl respect.}, $(\alpha'_p, \theta'_p)$) be a compatible defining couple of $T_p^\C\partial D$ ({\sl respect.}, $T_q^\C\partial D'$). Let $f:D\to D'$ be holomorphic. Suppose
\begin{equation}\label{Ju-la}
\lambda_p:=\inf_{q\in \partial D'}\sup_{z\in D}\frac{\Omega^{\alpha_p}_{D,p}(z)}{\Omega^{\alpha'_q}_{D',q}(f(z))}<+\infty.
\end{equation}
Then there exists a unique point $q\in \partial D'$ such that $f$ has non-tangential limit $q$ at $p$, the following maps are bounded on every cone in $D$ with vertex at $p$:
\begin{enumerate}
\item $d(\tilde\rho^{D'}_q \circ f)_z(v_p^D)$,
\item $|1-\tilde\rho^D_p(z)|^{1/2}d(f-\rho_q^{D'}\circ f)_z(v_p^D)$,
\item $|1-\tilde\rho^D_p(z)|^{-1/2}d(\tilde\rho_q^{D'}\circ f)_z(\tau_p)$,
\item $d(f-\rho^{D'}_q \circ f)_z(\tau_p)$,
\end{enumerate}
where $\tau_p$ denotes any complex tangent vector to $\partial D$ at $p$. Moreover, the map (1) has non-tangential limit $\lambda_p$ and the maps (2) and (3) have non-tangential limit $0$ at $p$.
\end{theorem}
\begin{remark}
Using Abate's ``projection devices'' \cite{A0, A1}, one can replace the ``non-tangential approach'' in the statement with the bigger class of  {\sl restricted $K$-limits}. However, we are not going to develop this argument in this paper. 
\end{remark}
\begin{proof}
By Theorem~\ref{Julia}, there exists a (unique) point $q$ such that 
\[
\lambda_p=\sup_{z\in D}\frac{\Omega^{\alpha_p}_{D,p}(z)}{\Omega^{\alpha'_q}_{D',q}(f(z))},
\]
and  $f$ has non-tangential limit $q$ at $p$.

As in Subsection~\ref{subset-tecn} we can and we will assume that (H1) holds for $D$ at $p$ and $D'$ at $q$ (and we denote by $B'\subset D'\subset W'$ the two strongly convex domains entrapping $D'$). It is easy to see that this assumption does not change the statements.

In order to avoid burdening notations, we write $\Omega_{X}$ instead of $\Omega_{X,p}^{\alpha_p}$, for $X=B, D$ and $\Omega_{X}$ instead of $\Omega_{X,q}^{\alpha'_q}$ for $X=D', W'$.

Then we define $g:=f|_B:B\to W'$. By \eqref{Eq:dineq} and \eqref{Ju-la}, we have
\[
\Omega_{W'}(g(z))\leq \Omega_{D'}(f|_{B}(z))\leq\frac{1}{\lambda_p}\Omega_{D}(z)\leq \frac{1}{\lambda_p}\Omega_{B}(z).
\]
Let
\[
\tilde\lambda_p:=\sup_{z\in B}\frac{\Omega_{B}(z)}{\Omega_{W'}(g(z))}.
\]
Hence, $\tilde\lambda_p\leq \lambda_p$. Choose now $z_0=\varphi^D_p(0)$ and $z_0'=\varphi^{D'}_q(0)$. By \cite[Proposition 2.3]{BCD} (actually, that proposition is proved for self-maps of the same domain, but it is easy to see that the proof adapts to holomorphic maps between different bounded, strongly convex domains with smooth boundary)
\[
\frac{1}{2}\log \tilde\lambda_p=\liminf_{z\to p}[k_B(z,z_0)-k_{W'}(f(z),z_0')],
\]
where $k_M$ denotes the Kobayashi distance of a domain $M$. Hence, we can apply Abate's Julia-Wolff-Carath\'eodory's theorem for strongly convex domains (see \cite[Thm. 0.8]{A0} and \cite[Thm. 2.1]{A1}), and we have that the maps (1)---(4) are bounded in any cone in $B$ with vertex at $p$. Since $B$ and $D$ are tangent at $p$, the same statement follows in $D$. Moreover, (2) and (3) have non-tangential limit $0$ at $p$. And (1) has non-tangential limit $\tilde\lambda_p$.  

We are left to show that $\tilde\lambda_p=\lambda_p$. 

By Proposition~\ref{local-Julia}, for every $0<a<1$ there exists a smooth curve $\gamma:[0,1]\to D\cup\{p\}$ converging to $p$ such that $\gamma'(1)\not\in T_p\partial D$ and 
$\liminf_{t\to 1}\frac{\Omega_D(\gamma(t))}{\Omega_{D'}(f(\gamma(t)))}\geq a\lambda_p$. 

By Proposition~\ref{Prop:M2}, we have
\begin{equation}\label{Eq:great1}
\lim_{t\to 1}\frac{\Omega_{B}(\gamma(t))}{\Omega_{D}(\gamma(t))}=1.
\end{equation}

By \cite[Corollary 1.8]{A0}, (recalling that $\tilde\rho_q^{D'}$ is also the Lempert projection of $W'$ onto $\varphi_q^{D'}$), we know that $\tilde\rho_q^{D'}(f(\gamma(t)))$ converges to $1$ non-tangentially. Since $\varphi_q^{D'}(\oD)$ is transverse to $\partial D'$ at $q$, this implies that if $v$ is the limit as $n\to \infty$ of  $\frac{(f(\gamma(t_n))'}{|(f(\gamma(t_n))'|}$ for some sequence $t_n\to 1$, then either $v\not\in T_q\partial D'$ or $v\in T_q^\C\partial D'$. Hence, by Proposition~\ref{Prop:M2}, we have
\[
\lim_{t\to 1}\frac{\Omega_{D'}(f(\gamma(t)))}{\Omega_{W'}(f(\gamma(t)))}=1.
\]
By the previous equation and \eqref{Eq:great1} we have
\[
\tilde\lambda_p\geq \liminf_{t\to 1}\frac{\Omega_B(\gamma(t))}{\Omega_{W'}(g(\gamma(t)))}=\liminf_{t\to 1}\frac{\Omega_D(\gamma(t))}{\Omega_{D'}(f(\gamma(t)))}\frac{\Omega_{D'}(f(\gamma(t)))}{\Omega_{W'}(f(\gamma(t)))}\frac{\Omega_{B}(\gamma(t))}{\Omega_{D}(\gamma(t))}\geq a\lambda_p.
\]
By the arbitrariness of $a$, we have $\tilde\lambda_p\geq \lambda_p$, and we are done.
\end{proof}

It is interesting to compare Theorem~\ref{Thm:JWC} with Abate's Julia-Wolff-Carath\'eodory theorem for strongly pseudoconvex domains~\cite[Theorem 0.2]{A1}: 

\begin{theorem}[Abate]\label{AbateJWC}
Let  $D\subset \C^n, D'\subset \C^m$ be bounded, strongly pseudoconvex domain with smooth boundary. Let $p\in \partial D$. Let $f:D\to D'$ be holomorphic. Suppose
\begin{equation}\label{dist-Abate}
\liminf_{z\to p} \frac{\hbox{dist}(f(z), \de D')}{\hbox{dist}(z, \de D)}<+\infty.
\end{equation}
Then there exists a unique point $q\in \partial D'$ such that $f$ has non-tangential limit $q$ at $p$, the following maps are bounded on every cone in $D$ with vertex at $p$:
\begin{enumerate}
\item $\pi_q(df_z(\nu_p))$,
\item $\hbox{dist}(z, \de D)^{1/2}d(f-\overline{\pi_q}\circ f)_z(\nu_p)$,
\item $\hbox{dist}(z, \de D)^{-1/2}\pi_q(df_z(\tau_p))$,
\item $d(f-\overline{\pi_q}\circ f)_z(\tau_p)$,
\end{enumerate}
where $\nu_p$ is the outer unit normal vector of $\partial D$ at $p$ (with respect to the Hermitian product in $\C^n$), $\nu_q$ is the unit outward normal of $\partial D$ at $q$ and $\pi_q(v):=\langle v, \nu_q\rangle\nu_q$ and $\tau_p$ denotes any complex tangent vector to $\partial D$ at $p$. Moreover, the map (1) has finite nonzero non-tangential limit  and the maps (2) and (3) have non-tangential limit $0$ at $p$.
\end{theorem}

Again, in the statement, non-tangential approach can be replaced with $K$-regions approach. 

Using Lempert's special  coordinates, one can take $\nu_p=v_p^D$ and $\pi_q=\rho_q^{D'}$, so that, the conclusions of Theorem~\ref{Thm:JWC} and Theorem~\ref{AbateJWC} are almost the same. The main difference between  Theorem~\ref{Thm:JWC} and Theorem~\ref{AbateJWC} is that the value of the normal projection of the derivative along the normal direction (when such directions are chosen to be ``compatible'') can be computed in terms of an intrinsic data, the number $\lambda_p$, as in the case of strongly convex domains.

Using estimates of the Kobayashi distance,  it is easy to see that \eqref{dist-Abate} is equivalent to 
\begin{equation}\label{Kob-Abate}
\frac{1}{2}\log \beta_f:=\liminf_{z\to p}[k_D(z_0, z)-k_{D'}(f(z), z_0')]<+\infty
\end{equation}
where $z_0\in D$ and $z_0'\in D'$ are two fixed points. 

As we remarked in the proof of Theorem~\ref{Thm:JWC}, in the strongly convex case (when suitably choosing $z_0, z_0'$), $\beta_f=\lambda_p$ (defined as in \eqref{Ju-la}), and this is also the value of the projection of the derivative of $f$ along the chosen complex geodesic at $p$. In strongly pseudoconvex domains however, there is no known relation between $\beta_f$ and the value of the normal projection of the differential of $f$ along the normal direction. Here we prove that, however, Abate's hypothesis \eqref{dist-Abate} is equivalent to~\eqref{Ju-la}:

\begin{proposition}
Let  $D\subset \C^n, D'\subset \C^m$ be bounded, strongly pseudoconvex domain with smooth boundary. For each $p\in \partial D$ ({\sl respectively}, $q\in \partial D'$) let $(\alpha_p, \theta_p)$ ({\sl respect.}, $(\alpha'_p, \theta'_p)$) be a compatible defining couple of $T_p^\C\partial D$ ({\sl respect.}, $T_q^\C\partial D'$). Let $z_0\in D$ and $z_0'\in D'$. Let $f:D\to D'$ be holomorphic. Then the following are equivalent:
\begin{enumerate}
\item $\lambda_p<+\infty$,
\item $\liminf_{z\to p}[k_D(z_0, z)-k_{D'}(f(z), z_0')]<+\infty$,
\item $\liminf_{z\to p} \frac{\hbox{dist}(f(z), \de D')}{\hbox{dist}(z, \de D)}<+\infty$.
\end{enumerate}
\end{proposition}
\begin{proof}
We already saw that (2) and (3) are equivalent. In particular, the choice of $z_0, z_0'$ is irrelevant for (2).  Now we use the same notations as in the proof of Theorem~\ref{Thm:JWC}, setting $z_0=\varphi_p^D(0)$ and $z_0'=\varphi_q^{D'}(0)$. We saw that $\lambda_p=\tilde\lambda_p$. Hence, it is enough to show that (2) is equivalent to $\tilde\lambda_p<+\infty$. 

On the one side, for all $z\in B$, since $B\subset D$ and $D'\subset W'$, 
\[
k_D(z_0, z)-k_{D'}(f(z), z_0')\leq k_B(z_0, z)-k_{W'}(f(z), z_0'),
\]
and hence, (1) implies (2). 

On the other side, if (2) holds, it follows from the proof of Theorem~\ref{AbateJWC} that there exists a sequence $\{z_n\}\subset D$ converging to $p$ such that $\{f(z_n)\}$ converges to $q$ (in fact, {\sl a posteriori}, any sequence converging non-tangentially to $p$ does) and 
\[
\limsup_{n\to\infty}[k_D(z_0, z_n)-k_{D'}(f(z_n), z_0')]<+\infty.
\]
By \cite[Lemma 5.4]{BG}, we can find $T\geq 1$ such that for $n$ sufficiently large so that $z_n$ stays sufficiently close to $p$ and $f(z_n)$ stays sufficiently close to $q$, we have
\[
|k_B(z_n,z_0)-k_D(z,z_0)|+|k_{W'}(f(z_n), z_0')-k_{D'}(f(z_n), z_0')|\leq \log T.
\]
Hence (2) implies $\tilde\lambda_p<+\infty$---and hence (1) holds.
\end{proof}

\section{Further properties of the pluricomplex Poisson kernel}

In this section we are going to prove some further property of the pluricomplex Poisson kernel in $\C^n$, such as uniqueness and (semi)continuity properties with respect to the change of pole, that will be useful later on.

We start by a uniqueness result, whose proof is exactly the same as that of \cite[Thm. 7.1]{BPT}:

\begin{proposition}\label{Prop:unique}
Let $D\subset \C^n$ be a bounded strongly pseudoconvex domain with smooth boundary and let $p\in \partial D$. Fix a defining couple  $(\alpha_p, \theta_p)$ of $T_p^\C\partial D$. Let $u$ be a maximal plurisubharmonic function on $D$ such that $\lim_{x\to q}u(x)=0$ for all $q\in \partial D\setminus\{p\}$ and
\[
\lim_{z\to p}\frac{\Omega^{\alpha_p}_{D,p}(z)}{u(z)}=1.
\]
Then $u=\Omega^{\alpha_p}_{D,p}$.
\end{proposition}

\begin{remark}
As observed in \cite[Prop. 4.3]{HW}, if $D$ is strongly convex, the previous uniqueness result holds if one replaces unrestricted limits with non-tangential limits. The proof of this result relies on complex geodesics and it is not clear how to extend to strongly pseudoconvex domains.
\end{remark}

Next, we prove that the pluricomplex Poisson kernel can be used to define a measure on the boundary:

\begin{proposition}\label{semicontinuityinp}
Let $D\subset \C^n$ be a bounded strongly pseudoconvex domain with smooth boundary. Choose a defining couple  $(\alpha_p, \theta_p)$ of $T_p^\C\partial D$ which varies  continuously with respect to $p$. Then, the function $\Omega^{\alpha_p}_{D,p}(z)$ is
upper semicontinuous with respect to the variable $p\in
\partial D$.
\end{proposition}
\begin{proof} Since $D$ is relatively compact with smooth boundary, there exists $r>0$ small enough such that the ball
$\B_p(r)$ of radius $r$ internally tangent at $D$ in $p$ is
contained in $D$, for each $p\in\partial D$.

Let us define $g:D\times \partial D\to\mathbb R_{\leq0}$ as
\begin{equation}\label{g-moves}
g(z,p)=\begin{cases}\Omega_{\B_p(r),p}(z) & \emph{if } z\in \B_p(r)\\
0 & \emph{if } z\not\in \B_p(r)\end{cases}
\end{equation}
The function  $g(z,p)$ is
continuous in both variables. Moreover
$\Omega_{D,p}^{\alpha_p}(z)\leq g(z,p)$ by Lemma~\ref{palla-equiv}.

Let us fix a point $q\in\partial D$ and a sequence
$\{q_n\}\subset\partial D$, $q_n\to q$. The
functions $\Omega_{D,q_n}^{\alpha_{q_n}}(z)$ are negative,
hence uniformly bounded from above. Taking the maximum limit
for $n\to \infty$, in $\Omega_{D,q_n}^{\alpha_{q_n}}(z)\ \leq\
g(z,q_n)$
 we get
$$\limsup_{n\to\infty}\Omega_{D,q_n,}^{\alpha_{q_n}}(z)\ \leq\ g(z,q).$$
Therefore the upper semicontinuous regularization with respect
to $z$ satisfies
$$F_q(z)=(\limsup_{n\to\infty}\Omega_{D,q_n}^{\alpha_{q_n}}(z))^*\ \leq\ (g(z,q))^*\ =\ g(z,q).$$

Hence $F_q(z)\in\mathcal S_{\alpha_q}(D)$, thus
$$\limsup_{n\to\infty}\Omega_{D,q_n}^{\alpha_{q_n}}(z)\ \leq\ F_q(z)\ \leq\ \Omega_{D,q}^{\alpha_q}(z).$$
Namely $\Omega^{\alpha_p}_{D,p}$ is upper semicontinuous in $q$.
\end{proof}

\begin{proposition}\label{boundedmeas}
Let $D\subset \C^n$ be a bounded strongly pseudoconvex domain  with smooth boundary. Choose a defining couple  $(\alpha_p, \theta_p)$ of $T_p^\C\partial D$ which varies  continuously with respect to $p$.  Let $K\subset\subset D$ be a compact set. Then the
function $|\Omega^{\alpha_p}_{p,D}(z)|$ is uniformly bounded
in $p\in\partial D$ with respect to $z\in K$.
\end{proposition}
\begin{proof}
By Forn\ae ss' embedding theorem \cite[Theorem 9]{Fo} there exist $N\geq n$, a smooth bounded strongly convex domain $D'\subset \mathbb C^N$ and a holomorphic map $F:\C^n\to \C^N$ such that $F$ is a biholomorphism on the image, which is a closed subvariety of $\C^N$, $F(D)\subset D'$,   $F(\partial D)\subset \partial D'$ and $F(\C^n)$ intersects transversally $\partial D'$.

Let $q\in \partial D'$, and let $\nu'_q$ be the outer unit normal
 vector at $\partial D'$ in $q$.  
 
 The transversality in Forn\ae ss' embedding theorem allows to define a defining couple $(\alpha_p^F, \theta_p^F)$ for $T_p\partial D$ by setting
\[
\theta_p^F(v):=\langle dF_p(v), \nu'_{F(p)}\rangle,
\]
where $\langle\cdot,\cdot\rangle$ is the standard hermitian product in $\mathbb C^N$,
and $\alpha_p^F:=\pm\Im  (\theta_p^F|_{T_p\partial D})$, where the sign is chosen so that $a_p \alpha_p^F=\alpha_p$ for some $a_p>0$.   Note that $(\theta_p^F, \alpha_p^F)$ varies continuously in $p$. Since also $\alpha_p$ varies continuously in $p$ and $\partial D$ is compact, it turns out that $a_p$ is continuous in $p$ and there exists $A>0$ such that for all $p\in \partial D$,
\[
a_p\leq A.
\]
 
 Let us define
$\phi:D'\times\partial D'\to\C$ as
$$\phi(w,q)\ =\ \langle w-q,\nu'_q\rangle,\ \ \ \forall (w,q)\in D'\times\partial D',$$

The function $h: D\times\partial D\to \mathbb C$,
$$h(z,p)\ =\ \exp(\phi(F(z),F(p))), \ \ \ \forall (z,p)\in D\times\partial D$$
is continuous in $p$, and ---for each fixed $p\in\partial D$---
is a strong peak function in $p$ for the domain $D$,  $\mathcal
C^1$-smooth up to the boundary. Moreover since $D'$ is strongly
convex, it follows that for all $z\in D$ and $p\in \de D$
\[
\Re\la F(z)-F(p),\nu'_{F(p)}\ra < 0,
\]
and hence $h(z,p)\in \D$ for all $z\in D$ and $p\in \de D$ and
$h(p,p)=1$.

Hence, for each fixed $p\in \partial D$, $P_{\D}(h(z,p))$ is a negative plurisubharmonic function in $D$. Moreover, let  $\gamma:[0,1]\to D\cup \{p\}$ be a $C^\infty$ curve such that $\gamma(t)\in D$ for all $t\in [0,1)$, $\gamma(1)=p$  and $\gamma'(1)\not\in T_p\partial D$. Then
\begin{equation*}
\begin{split}
\lim_{t\to 1} P_\D(h(\gamma(t), p))(1-t)&=\lim_{t\to 1}-\Re \left(\frac{1+\exp(\langle F(\gamma(t))-F(p), \nu'_{F(p)}\rangle)}{1-\exp(\langle F(\gamma(t))-F(p), \nu'_{F(p)}\rangle)}   (1-t)\right)\\&=-2\Re \frac{1}{\langle dF_p(\gamma'(1)),\nu'_{F(p)}\rangle},
\end{split}
\end{equation*}
Therefore, $P_{\D}(h(\cdot ,p))\in \mathcal S_{\alpha^F_p}(D)$ and, by Proposition~\ref{ovvia}, $a_p P_{\D}(h(\cdot ,p))\in \mathcal S_{\alpha_p}(D)$ for all $p\in \partial D$. Hence, 
$$ a_p P_{\D}(h(z,p)) \leq \Omega_{p,D}^{\alpha_p}(z) \leq\ 0.$$

By continuity, there exists a positive constant $M_K>0$ such
that $|P_\D \circ h(z,p)|\leq M_K$ for all $z\in K$ and $p\in\partial D$. Thus
$$|\Omega_{p,D}^{\alpha_p}(z)|\ \leq\ a_p|P_{\mathbb D}(h(z,p))| \leq a_pM_K\leq A M_K,\quad \forall z\in K,$$
proving the statement.
\end{proof}

\section{Pluricomplex Poisson kernel vs Pluricomplex Green function}

The aim of this section is to relate the pluricomplex Poisson kernel with the pluricomplex Green function of a bounded strongly
pseudoconvex domain in $\C^n$.

Recall that, given a hyperconvex bounded domain $D\subset \C^n$ the {\sl pluricomplex Green function} of $D$ with pole $z\in D$ (see, {\sl e.g.}, \cite{Kl}) is 
\[
G_D(z,w):=\sup\{u(w): u<0, u\in \hbox{psh}(D), \limsup_{w\to z}[u(w)-\log |w-z|]<+\infty\}.
\]
The function $G_D(z,\cdot)$, extended by $0$ on $\partial D$, is continuous  on $\overline{D}\setminus\{z\}$,  plurisubharmonic and maximal. Moreover, $G_D(\cdot, \cdot)$ is continuous (as function with values in  $[-\infty, 0]$) on
$D\times \overline{D}$ (see \cite[Th\'eor\`em~(0.6)]{De})

In case $D$ is smooth and strongly convex, Lempert \cite{Le1}  showed that $G_D$ is symmetric in $(z,w)$ and it is smooth on $\overline{D}\times \overline{D}\setminus \{(z,w): z=w\}$. In case $D$ is smooth and strongly pseudoconvex $G_D(z,\cdot)$ is in general not $C^2$ and it is in general not symmetric in $(z,w)$ (see \cite{BD}). Actually, the symmetry in $(z,w)$ is equivalent to the plurisubharmonicity of the function $G_D(\cdot, w)$ for all fixed $w\in D$.

 However, by results of Guan \cite{Gu} and B\l ocki \cite{Bl},  the
pluricomplex Green function $G_D(z,w)$ of a bounded strongly
pseudoconvex domain in $\C^n$ with pole at $z$ is $C^{1,1}$
with respect to $w\in \overline{D} \setminus\{z\}.$
In what follows we need this slight extension of the previous result: 
\begin{lemma}\label{stime}
Let $D\subset \C^n$ be a bounded strongly
pseudoconvex domain with smooth boundary. Let $p\in \partial D$ and denote by $\nu_p$ the outer unit  normal vector to $\partial D$ at $p$. For every compact set $K_1\subset\subset D$ there exist $h_0>0$ and $C>0$ such that 
\[
\left| \frac{\partial G_D(z,p)}{\partial \nu_p}-\frac{\partial G_D(z,p-t\nu_p)}{\partial \nu_p}  \right| \leq Ct
\]
for all $t\in [0,h_0]$ and $z\in K_1$.
\end{lemma}
\begin{proof}
We follow the argument in \cite{Bl2}.

Fix $z\in D$. For a given $\epsilon>0$ let $D_{\epsilon} := D \setminus \overline{B(z,\epsilon)}$, where $B(z,\epsilon)$ is the Euclidean ball of center $z$ and radius $\epsilon>0$.

For a fixed $z\in K_1$, B\l ocki in \cite{Bl2} proves that given $\epsilon>0$ small and $\delta\in (0,1)$ there exist  continuous functions $g_z^{\epsilon}(w)$ converging locally uniformly to $G_D(z,\cdot)$ in 
$D\setminus \{z\}$ as $\epsilon $ goes to zero (see \cite[Eq. (2.2)]{Bl2}) and that there exist  functions $g_z^{\epsilon,\delta}$ which are smooth on  $\overline{D_{\epsilon}}$ (see \cite[Prop. 2.2]{Bl2}) and     
uniformly converge to $g_z^{\epsilon}$ in $D_\epsilon$ as $\delta$ goes to zero (see  \cite[Eq. (2.3)]{Bl2}). Actually, in \cite{Bl},  the functions $g_z^{\epsilon}$ and $g_z^{\epsilon,\delta}$ are denoted by $g^\epsilon$ and $g^{\epsilon,\delta}$ (without the subscript $z$). For the sake of clarity, we prefer to indicate here also the corresponding pole.

Take $\tilde\epsilon>0$ such that $\tilde K_1:=\bigcup_{z\in K_1}\overline{B(z,\tilde\epsilon)}\subset\subset D$ and let $h_0>0$ be such that $p-t\nu_p\in D\setminus \tilde K_1$ for all $t\in (0,h_0]$.

\medskip

{\sl Claim:} There exist $\epsilon_0\in (0,\tilde\epsilon)$, $\delta_0\in (0,1)$ and  $C>0$  such that for all $t\in [0,h_0]$, $z\in K_1$, $\delta<\delta_0$ and $\epsilon<\epsilon_0$,
\begin{equation}\label{Blo1}
|\nabla g_z^{\epsilon, \delta}(p-t\nu_p)|+|\nabla^2 g_z^{\epsilon, \delta}(p-t\nu_p)|\leq C.
\end{equation} 
 
 \medskip

Assuming the claim for the moment, the proof ends as follows.  Let $\epsilon\in (0,\epsilon_0)$. 
The estimate \eqref{Blo1}  implies that the first and second derivatives of the functions $[0,h_0]\ni t\mapsto v_z^{\epsilon,\delta}(t):=g_z^{\epsilon,\delta}(p-t\nu_p)$ are uniformly bounded. By the Mean Value Theorem, this implies also that the $v_z^{\epsilon,\delta}$ and their first derivatives  are Lipschitz in $[0,h_0]$ with uniform constant $C>0$ (independent of $z\in K_1, \epsilon, \delta$). Since for any fixed $t\in (0,h_0]$, $g_z^{\epsilon,\delta}(p-t\nu_p)$ converges to $g_z^{\epsilon}(p-t\nu_p)$ as $\delta\to 0$, it follows that  also the $\{v_z^{\epsilon,\delta}\}$ are uniformly bounded. By Arzel\`a-Ascoli's Theorem, for $z\in K_1$ and $\epsilon$ fixed, up to extracting subsequences, we can assume that $\{v_z^{\epsilon,\delta}\}$  converges uniformly as $\delta\to 0$ in the $C^1$-topology of $[0,h_0]$ to a function $v_z^{\epsilon}:[0,h_0]\to \R$. Clearly,  $v_z^{\epsilon}$ and their first derivatives are uniformly bounded and Lipschitz in $[0,h_0]$, with uniform Lipschitz constant $C>0$.  Since $v_z^{\epsilon, \delta}(t)=g_z^{\epsilon,\delta}(p-t\nu_p)$ and these latter functions converge to $g_z^\epsilon(p-t\nu_p)$ as $\delta\to 0$, it follows that the functions $(0,h_0]\ni t\mapsto g_z^\epsilon(p-t\nu_p)$ can be extended $C^1$ on $[0,h_0]$ and they are, together with their first derivatives,  uniformly Lipschitz in $[0,h_0]$. 

Repeating the previous argument with the functions $v_z^\epsilon$ instead of $v_z^{\epsilon,\delta}$ and taking the limit for $\delta\to 0$, we see that also $G_D(z,p-t\nu_p)$ and its derivative with respect to $t$, that is $-\frac{\partial G_D(z,p-t\nu_p)}{\partial \nu_p}$, are uniformly Lipschitz in $t\in [0,h_0]$ independently of $z\in K_1$, and we are done.

We are left to prove the claim.  
In \cite[Thm. 1.1 and Thm. 3.1]{Bl2} it is proved that for any $z\in K_1$  there exist $\epsilon_z>0$, $\delta_z\in (0,1)$ and a constant $C_z>0$ such that  for all $\epsilon\in (0,\epsilon_z)$ and $\delta\in (0,\delta_z)$,
\begin{equation*}
\begin{split}
|\nabla g_z^{\epsilon, \delta}(w)|&\leq \frac{C_z}{|w-z|}  \quad \forall w\in D_\epsilon\\
|\nabla^2 g_z^{\epsilon, \delta}(w)|&\leq \frac{C_z}{|w-z|^2} \quad \forall w\in D_\epsilon.
\end{split}
\end{equation*}

Note that, if $\epsilon<\tilde\epsilon$ and $t\in [0, h_p]$, then $p-t\nu_p\in D_\epsilon$. Therefore the previous estimate holds in particular for $w=p-t\nu_p$, $t\in (0,h_0]$. Moreover, $\min\{|p-t\nu_p-z|: t\in [0,h_0], z\in K_1\}>0$. Hence, in order to prove \eqref{Blo1}, it is enough to show that there exist $C'>0$, $\delta_0\in (0,1)$ and $\epsilon_0\in (0,\tilde\epsilon)$ such that $\delta_z\geq \delta_0$, $\epsilon_z\geq \epsilon_0$ and $C_z\leq C'$ for all $z\in K_1$ and $w=p-t\nu_p$, $t\in (0,h_0]$. 

Analyzing B\l ocki's proof, one can check that the dependence of $\delta_z$, $\epsilon_z$ and $C_z$ on $z$ is via the distance of $z$ from $\partial D$ and  a constant $b(z)$ (defined and called just $b$ at \cite[pag. 348]{Bl2}) which is defined by
\[
b(z):=\liminf_{w\to \partial D}\frac{|G_D(z,w)|}{\hbox{dist}(w,\partial D)}.
\]
In particular, one can check that if $\inf_{z\in K_1}\hbox{dist}(z,\partial D)>0$ (which is the case since $K_1$ is compact in $D$) and $\inf_{z\in K_1}b(z)>0$ then there exist $C', \epsilon_0, \delta_0>0$ such that $C_z\leq C'$, $\delta_z\geq \delta_0$ and $\epsilon_z\geq \epsilon_0$ for all $z\in K_1$.

So we are left to show that $\inf_{z\in K_1}b(z)>0$. To this aim,  let $\bar{r}$ be a positive constant such that for all $p \in \partial D$ the open ball of radius 
$2 \bar{r}$  tangent to $\partial D$ in $p$ is contained in $D$.
Let  $T_{\bar{r}} :=  \{ w \in D : \hbox{dist}(w, \partial D) \geq \bar{r} \}$ and replace the $\gamma$ at pag. 348 in \cite{Bl2} with  $\gamma := \max_{ K_1 \times T_{\bar{r}}} G(z,w)$. 
Then, following again the argument in the proof of \cite[Theorem 1.1]{Bl2} at p. 348, we see that there exists $\beta>0$ such that    $b(z) \geq \beta > 0$ for any $z \in K_1$, and we are done.
     \end{proof}

\begin{lemma}\label{Lem:greatS}
Let $D\subset \C^n$ be a bounded strongly
pseudoconvex domain with smooth boundary. Let $\nu_p$ denote the outer unit  normal vector to $\partial D$ at $p$. Suppose that the pluricomplex Green function $G_D$ is symmetric,  then:
\begin{enumerate}
 \item    $D\ni z \mapsto  - \frac{\partial G_D(z,p)}{\partial \nu_p} $ is a  maximal plurisubharmonic function in $D$ for all  $p\in \partial D$.
\item For every $ p $ and $q$ in $  \partial D$ with $p \neq q,$  $\lim_{z\to q}  \frac{\partial G_D(z,p)}{\partial \nu_p} = 0. $ 
\end{enumerate}
\end{lemma}
\begin{proof}
(1)
Fix $p \in \partial D$. Since $G_D$ vanishes on the boundary of $D$, then 
\[
  - \frac{\partial G_D(z,p)}{\partial \nu_p}  =  \lim_{h\to 0^+}  \frac{G_D(z,p- h \nu_p)}{h}. 
 \] 
Now if $z$ varies on a compact set $K$ of $D$, by Lemma \ref{stime}, and by the Mean Value Theorem, for $h>0$ small, we have  
\[ 
\left| \frac{G_D(z,p- h\nu_p)}{h} + \frac{\partial G_D(z,p)}{\partial \nu_p} \right| = \left| -\frac{\partial G_D(z,p-t\nu_p)}{\partial \nu_p} + \frac{\partial G_D(z,p)}{\partial \nu_p}\right| \leq Ct, 
\] 
where $t\in (0,h)$. 

Hence the functions $z \mapsto    \frac{G_D(z,p- h \nu_p)}{h}$  converge locally uniformly in $z\in D$ to  $z\mapsto - \frac{\partial G_D(z,p)}{\partial \nu_p}$.
 Now for every fixed $ h>0$  the function $ D\ni z \mapsto G_D(z,p - h \nu_p)$ is maximal plurisubharmonic since we are assuming it is symmetric. From this and the locally uniformly convergence, (1) follows.

(2) Fix $q\in\partial D\setminus\{p\}$. By \cite[Theorem 9]{Fo} there exist $m\geq n$, a smooth bounded strongly convex domain $C\subset \C^m$ and a  holomorphic map  $\Phi: \C^n \to  \C^m$, so that $\Phi$ is a biholomorphism onto its image and $\Phi(\C^n)$ is a closed subvariety of $\C^m$. Moreover,  $\Phi(D)\subset C$, $\Phi(\partial D)\subset \partial C$ and $\Phi(\C^n)$ intersects transversally $\partial C$.

Let $\tilde p:=\Phi(p)$. Since $d\Phi_p(T_p\partial D)\subset T_{\tilde p}\partial C$ and $\Phi(\C^n)$ intersects transversally $\partial C$, it follows that $d\Phi_p(\nu_p)\not\in T_{\tilde p} \partial C$. Thus, up to an affine change of coordinates in $\C^m$, we can assume that   $d\Phi_p(\nu_p)=\nu_{\tilde p}$, where $\nu_{\tilde p}$ denotes the outer unit normal vector to $\partial C$ at $\tilde p$. 

Fix $h>0$ and let  $u(z,h):=G_C(\Phi(z), \Phi(p-h\nu_p))$ and $x_h:=p-h\nu_p$. Note that $u(\cdot ,h)$ is a negative plurisubharmonic function in $D$. Moreover, let $C:=\sup_{z\in D}\log\frac{\|\Phi(z)-\Phi(x_h)\|}{\|z-x_h\|}$. Note that $C<+\infty$ since $\Phi$ is holomorphic in $\C^n$. Then,
\begin{equation*}
\begin{split}
&\limsup_{z\to x_h}(u(z,h)-\log \|z-x_h\|)=\limsup_{z\to x_h}(u(z,h)-\log \|\Phi(z)-\Phi(x_h)\|\\&+\log\frac{\|\Phi(z)-\Phi(x_h)\|}{\|z-x_h\|})\\&\leq 
\limsup_{z\to x_h}(G_C(\Phi(z), \Phi(p-h\nu_p))-\log \|\Phi(z)-\Phi(x_h)\|+C)<+\infty,
\end{split}
\end{equation*}
where $\limsup_{z\to x_h}(G_C(\Phi(z), \Phi(p-h\nu_p))-\log \|\Phi(z)-\Phi(x_h)\|)<+\infty$ by the very definition of pluricomplex Green function. Since $G_D(\cdot, x_h)$ is the supremum of all negative plurisubharmonic functions in $D$ having at most a log-singularity at $x_h$, it follows that $u(z,h)\leq G_D(z, p-h\nu_p)$ for all  $z\in D$. Therefore,
\begin{equation*}
\begin{split}
 0\geq - \frac{\partial G_D(z,p)}{\partial \nu_p} &=\lim_{h\to 0}\frac{G_D(z, p-h\nu_p)}{h}\geq \lim_{h\to 0}\frac{G_C(\Phi(z), \Phi(p-h\nu_p))}{h}=- \frac{\partial G_C(\Phi(z),\tilde p)}{\partial \nu_{\tilde p}}.
\end{split}
\end{equation*}
Since, by \cite[Thm. 6.1]{BPT}, $- \frac{\partial G_C(\Phi(z),\tilde p)}{\partial \nu_{\tilde p}}=\Omega_{\tilde C, \tilde p}(\Phi(z))$ and hence 
$\lim_{z\to q}- \frac{\partial G_C(\Phi(z),\tilde p)}{\partial \nu_{\tilde p}}=0$, we are done.   
 \end{proof}

Now we relate the pluricomplex Green function with the pluricomplex Poisson kernel. First of all, notice that if $D\subset \C^n$ is a bounded strongly pseudoconvex domain in $\C^n$ with smooth boundary, one can choose at each point the outer unit normal vector $\nu_p$, and the map $\partial D\ni p\mapsto \nu_p$ is smooth. At each point $p\in\partial D$ we  can then choose the defining couple  for $T_p\partial D$ given by $\tilde\theta_p(v)=-i\langle v, \nu_p\rangle$, $v\in \C^n$, and $\tilde\alpha_p:=\Im \theta_p|_{T_p\partial D}$. These defining couples vary continuously with $p$. With this choice, we denote 
\[
\Omega_{D,p}:=\Omega_{D,p}^{\tilde\alpha_p}.
\]

\begin{proposition}\label{Prop:uguali}
Let $D\subset \C^n$ be a bounded strongly
pseudoconvex domain  with smooth boundary. For each $p\in \partial D$, let $\nu_p$ denote the outer unit normal vector to $\partial D$. Suppose that the pluricomplex Green function $G_D$ is symmetric.  Then for all $z\in D$,
\[
 - \frac{\partial G_D(z,p)}{\partial \nu_p}=\Omega_{D,p}(z).
\]
\end{proposition}
\begin{proof}
By Lemma~\ref{Lem:greatS} and Proposition~\ref{Prop:unique} we only need to show that 
\[
\lim_{z\to p}\frac{ - \frac{\partial G_D(z,p)}{\partial \nu_p}}{\Omega_{D,p}(z)}=1.
\]
Arguing as in Subsection~\ref{subset-tecn}, we can assume that $B\subset D\subset W$ where $B$ and $W$ are strongly convex domains with smooth boundaries and $\partial D$, $\partial B$ and $\partial W$ coincide near $p$.  Let $\Omega_{B,p}$ ({\sl respectively} $\Omega_{W,p}$) be the pluricomplex Poisson kernel of $B$ ({\sl resp.}, of $W$) associated to $(\tilde\alpha_p, \tilde\theta_p)$.

Fix a compatible defining couple $(\alpha',\theta')$ for $T_p^\C\partial D$.  By Proposition~\ref{ovvia}  there exists $c>0$ such that $\Omega_{D,p}=c \Omega^{\alpha'}_{D,p}$,  $\Omega_{B,p}=c \Omega^{\alpha'}_{B,p}$ and $\Omega_{W,p}=c \Omega^{\alpha'}_{W,p}$. 

Since $B\subset D\subset W$ for all $z,w\in B$, $z\neq w$, we have
\[
G_W(z,w)\leq G_D(z,w)\leq G_B(z,w).
\]
Since for $w\to p$ the pluricomplex Green functions of $B, D, W$ tends to $0$, we obtain for all $z\in B$,
\[
- \frac{\partial G_W(z,p)}{\partial \nu_p}\leq - \frac{\partial G_D(z,p)}{\partial \nu_p}\leq - \frac{\partial G_B(z,p)}{\partial \nu_p}.
\]
By \cite[Thm. 6.1]{BPT}, $- \frac{\partial G_B(\cdot,p)}{\partial \nu_p}=\Omega_{B,p}$ and $- \frac{\partial G_W(\cdot,p)}{\partial \nu_p}=\Omega_{W,p}$. Hence,  by \eqref{Eq:dineq}, for all $z\in B$,
\[
\frac{ \Omega_{B,p}(z)}{\Omega_{W,p}(z)}\leq \frac{ - \frac{\partial G_D(z,p)}{\partial \nu_p}}{\Omega_{D,p}(z)}\leq \frac{ \Omega_{W,p}(z)}{\Omega_{B,p}(z)}.
\]
Taking the limit for $z\to p$, the result follows  from Proposition~\ref{Prop:M2}.
\end{proof}

\section{Reproducing formula}

We briefly recall Demailly's construction \cite{De0, De} for the reproducing formula of pluri(sub)harmonic functions in terms of the pluricomplex Green function.

Let $D\subset \C^n$ be a bounded
strongly pseudoconvex domain with smooth boundary. Let $\psi$ be a defining function of $D$ and define
\[
\omega_{\de D}:=\frac{(dd^c\,\psi)^{n-1} \wedge
d^c\psi}{|d\psi|^n}|_{\partial D},
\]
where $d^c\psi=i(\overline\partial - \partial)\psi$. The form $\omega_{\de D}$ is a positive $(2n-1)$-real form, independent of the  function $\psi$ chosen to define it.

\begin{remark}
The form $\omega_{\partial D}$ can be also expressed in terms of the Levi form, $\hbox{Levi}(\psi)$, of $\psi$:
\begin{equation}\label{Eq:Levi-Dem}
\omega_{\partial D} =\ 4^{n-1} (n-1)! \frac{\det(\hbox{Levi}\,\psi)}{|d\psi|^{n-1}}\, d \hbox{Vol}_{\partial D}.
\end{equation}
In order to prove such a formula,  fix an orthonormal basis $(v_1,\ldots,v_{n-1})$ of $T_p^{\mathbb C}\partial D$, and let, as usual, $\nu_p$ be the outer unit normal vector of $\partial D$ at $p$. Fix the orientation of $\mathbb C^n$  given by the orthonormal basis $(v_1,Jv_1,\ldots,v_{n-1},Jv_{n-1},\nu_p,J\nu_p)$. With such a choice, the set $\mathcal B=(v_1,Jv_1,\ldots,v_{n-1},Jv_{n-1},J\nu_p)$ is a positive-oriented orthonormal basis of $T_p\partial D$.

If $p\in\partial D$, $v\in T_p\partial D$
$$d^c\psi(Jv)=i\sum_j \left[\frac{\partial\psi}{\partial \overline{z}_j}(-i\overline{v_j})-\frac{\partial\psi}{\partial z_j}(iv_j)\right]=\sum_j \left[\frac{\partial\psi}{\partial \overline{z}_j}\overline{v_j}+\frac{\partial\psi}{\partial z_j}v_j\right]=(d\psi)_p(v)\,.$$

Let us compute the form $\omega_{\partial D}$ in the basis $\mathcal B$:

$$\frac{(dd^c\psi)^{n-1}\wedge d^c\psi}{|d\psi|^n}(v_1,Jv_1,\ldots,v_{n-1},Jv_{n-1},J\nu_p) = \frac{(dd^c\psi)^{n-1}}{|d\psi|^{n-1}}(v_1,Jv_1,\ldots,v_{n-1},Jv_{n-1})\,.$$

Let   $w_j := \frac12(v_j-iJv_j)$ and let $\beta_s$ be the $(1,0)$-form such that
$$\beta_s\left(w_j \right)=\delta_{js}\,.$$
Let $\lambda_1, \ldots, \lambda_{n-1}$ be the eigenvalues and $v_1,\ldots,v_{n-1}$ be  a diagonalizing orthonormal basis of the Hermitian form  
$$\left(\frac{\partial^2\psi}{\partial z_i\partial \overline{z}_j}\right)_p(u_1;u_2), \ \ \  u_1,u_2\in T^{\mathbb C}\partial D\,.$$
Then $dd^c\psi=2i\partial\overline\partial\psi$ and so
$$dd^c\psi|_{T^{\mathbb C}_p\wedge T^{\mathbb C}_p} = \sum_{j=1}^{n-1}2i \lambda_j  \beta_j\wedge \overline \beta_j\,.$$
Hence
$$\left(dd^c\psi|_{T^{\mathbb C}_p\wedge T^{\mathbb C}_p}\right)^{n-1}\ =\ \sum_{j_1,\ldots,j_{n-1}}\lambda_{j_1}\cdots\lambda_{j_{n-1}}(2i)^{n-1} (\beta_{j_1}\wedge \overline{\beta} _{j_1}) \wedge \cdots \wedge (\beta_{j_{n-1}}\wedge \overline{\beta} _{j_{n-1}}),.$$
Each term is a wedge product of $(1,1)$-forms, hence those forms commute and we can write
$$\left(dd^c\psi|_{T^{\mathbb C}_p\wedge T^{\mathbb C}_p}\right)^{n-1}\ = (n-1)! (2i)^{n-1} \lambda_1\cdots \lambda_{n-1} \beta_1\wedge\overline\beta_1\wedge\cdots\wedge\beta_{n-1}\wedge\overline\beta_{n-1}.$$
Since
$$2i(\beta_k\wedge\overline\beta_k)(v_h,Jv_h)\ =\ 2i \beta_k\wedge\overline\beta_k\left(w_h+\overline w_h;\frac{w_h-\overline w_h}{- i}\right)\ =\ 4\delta_{h,k},$$
equation \eqref{Eq:Levi-Dem} follows.
\end{remark}

Let $\varphi$ be a negative plurisubharmonic exhaustion function in $D$ such that $\exp(\varphi)$ is continuous on $\overline{D}$ and $\varphi=0$ on $\partial D$. Let $r<0$ and let $B(r):=\{z\in D: \varphi(z)<r\}$. Let $\varphi_r(z)=\max\{\varphi(z), r\}$. Hence, $(dd^c\varphi_r)^n=\chi_{\C^n\setminus B(r)}(dd^c\varphi)^n+\mu_{\varphi,r}$, where $\chi_{\C^n\setminus B(r)}$ is the characteristic function of $\C^n\setminus B(r)$ and $\mu_{\varphi,r}$ is a positive measure supported on $\partial B(r)$. If $\int_D(dd^c\varphi)^n<+\infty$ then $\mu_{\varphi, r}$ weakly converges to a positive measure $\mu_\varphi$ supported on $\partial D$ as $r\to 0$ and whose total mass is $\int_D(dd^c\varphi)^n$. 

Demailly \cite[Th\'eor\`eme 5.1]{De} proved the following representative formula: if $f$ is a plurisubharmonic function in $D$, continuous in $\overline{D}$, then
\begin{equation}\label{Dem-repr}
f(z)=\frac{1}{(2\pi)^n}\mu_{G_D(z,\cdot)}(f)-\frac{1}{(2\pi)^n}\int_{w\in D} |G_D(z, w)| dd^c f(w)\wedge (dd^c G_D(z,w))^{n-1}.
\end{equation}

Demailly's formula holds, in fact, for hyperconvex bounded domains.

We first show that for strongly pseudoconvex domains the measure $\mu_z$ is related to the pluricomplex Green function:

\begin{lemma}\label{Lem:Green-Demailly}
Let $D\subset \C^n$ be a bounded strongly pseudoconvex domain
with smooth boundary. Let $G_D(z,w)$ be the pluricomplex Green
function of $D$ with pole in $z\in D$. Let $\nu_p$ denote the outer unit normal vector
to $\de D$ at $p\in \de D$. Then
\begin{equation}\label{espr}
\mu_{G_D(z,\cdot)}= \left(\frac{\partial G_D(z,\cdot)}{\partial \nu_p}\right)^n \omega_{\partial D}.
\end{equation}
\end{lemma}

\begin{proof}
Let $\varphi\in\mathcal C^2(\overline D)$ be a strictly
plurisubharmonic defining function for $D$, {\sl i.e.}
$\varphi$ is strictly plurisubharmonic and negative on $D$, and
vanishes on $\partial D$. 

Fix $z\in D$. Let $p\in \partial D$. Since both $\varphi$ and $G_D(z,\cdot)$ vanish on $\partial D$, then
$d \varphi|_{T_{p}\partial D}=0$, $d
G_D(z,\cdot)|_{T_{p}\partial D}=0$. 
Therefore, there exist real numbers $a, b$ such that  $d
G_D(z,\cdot)|_{w=p}(v)=a\Re \langle  \nu_p, v\rangle$  and $d\varphi_p(v)=b\Re \langle  \nu_p, v\rangle$ for all $v\in \C^n$ (where, as usual, $\langle \cdot, \cdot \rangle$ denotes the standard Hermitian product in $\C^n$).
Since
\[
dG_D(z,\cdot)|_{w=p}(\nu_p)= \frac{\de G_D(z,\cdot )}{\de \nu_p}, \quad \v_p(\nu_p) = \frac{\de \v(p)}{\de \nu_p},
\]
and by Hopf's Lemma $\frac{\de G_D(z,\cdot )}{\de \nu_p}\neq 0$ and $\frac{\de \v(p)}{\de \nu_p}\neq 0$, it follows that for every $\zeta\in D$,
$$
\frac{G_D(z,w)}{\varphi(w)}=\frac{dG_D(z,\cdot)|_{w=p} (\frac{\zeta-p}{|\zeta-p|})+o(1)}{d\varphi_p( \frac{\zeta-p}{|\zeta-p|})+o(1)}=\frac{\frac{\de G_D(z,\cdot )}{\de \nu_p}\Re\langle \nu_p ,\frac{\zeta-p}{|\zeta-p|}\rangle +o(1)}{\frac{\de \v(p)}{\de \nu_p}\Re \langle\nu_p , \frac{\zeta-p}{|\zeta-p|}\rangle+o(1)}.
$$
Hence, 
\begin{equation}\label{limGv}
\lim_{w\to p}\frac{G_D(z,w)}{\varphi(w)}\ =\
\frac{\frac{\partial G_D(z,\cdot)}{\partial
\nu_p}(p)}{\frac{\partial\varphi}{\partial \nu_p}(p)}.
\end{equation}

Let us now consider the measure on $\partial D$
$$\mu_\varphi\ =\ (dd^c\,\varphi)^{n-1}\wedge d^c\varphi|_{\partial D}.$$

By \cite[Th\'eor\`eme 3.8]{De} and \eqref{limGv}, it follows immediately that
$$\mu_{G_D(z,\cdot)}=\left.\frac{\left(\frac{\partial G_D(z,\cdot)}{\partial \nu_p}(p)\right)^n}
{\left(\frac{\partial\varphi}{\partial \nu_p}(p)\right)^n}\right|_{\partial D}\ \mu_\varphi.$$

Since   $|d\varphi_p|=\frac{\partial\varphi}{\partial
\nu_p}(p)$, we are done.
\end{proof}

In case the pluricomplex Green function is symmetric, by Proposition~\ref{Prop:uguali} and Lemma~\ref{Lem:Green-Demailly} we have

\begin{theorem}\label{Trepfor} Let $D\subset \C^n$ be a bounded strongly pseudoconvex domain with smooth boundary.
Assume that the pluricomplex Green function of $D$ is symmetric. Then for every
plurisubharmonic function  $f$ in $D$ continuous up to the
boundary, 
\begin{equation*}
\begin{split}
f(z) &= \frac{1}{(2\pi)^n}\int_{\partial D} f(\xi) |\Omega_{D,\xi}(z)|^n\omega_{\partial D}(\xi)\\&-\frac{1}{(2\pi)^n}\int_{w\in D} |G_D(z, w)| dd^c f(w)\wedge (dd^c G_D(z,w))^{n-1}.
\end{split}
\end{equation*}
In particular, if $f$ is pluriharmonic in $D$ and continuous on $\overline{D}$, 
\begin{equation*}
f(z) = \frac{1}{(2\pi)^n}\int_{\partial D} f(\xi) |\Omega_{D,\xi}(z)|^n\omega_{\partial D}(\xi).
\end{equation*}
\end{theorem}

In the same formula for strongly convex domains (see, \cite[Thm 8.2]{BPT}) there is a missing factor $1/(2\pi)^n$ in front of the first integral.  

We do not know if the previous formula holds in case $G_D$ is not symmetric.

\end{document}